\documentclass[11pt,a4paper]{amsart}
\usepackage[utf8]{inputenc}
\usepackage[english]{babel}
\usepackage{amsmath}
\usepackage{amsfonts}
\usepackage{amsthm}
\usepackage{mathptmx}
\usepackage{color}
\usepackage{amssymb}
\usepackage{mathrsfs, ccfonts, mathpazo, mathrsfs}
\usepackage{enumitem,linegoal}
\usepackage{tikz}
\usepackage{dsfont}
\usepackage{enumitem} 
\usepackage{calc}
\usepackage{stmaryrd}
\usepackage[mathscr]{eucal}
\usepackage{xcolor}
\usepackage{hyperref}
\definecolor{mediumtealblue}{rgb}{0.0, 0.33, 0.71}
\definecolor{tangelo}{rgb}{0.98, 0.3, 0.0}
\hypersetup{%
	colorlinks=true, 
	linkcolor= mediumtealblue,
	anchorcolor=black, 
	citecolor= red, 
	urlcolor = blue
}
\usepackage[capitalize]{cleveref}
\usepackage[left=2cm,right=2cm,top=2cm,bottom=2cm]{geometry}
\newtheorem{theorem}{Theorem}[section]
\newtheorem{defn}[theorem]{Definition}
\newtheorem{lem}[theorem]{Lemma}

\newtheorem{prop}[theorem]{Proposition}
\newtheorem{cor}[theorem]{Corollary}

\newtheorem{algorithm}{Theorem}[]
\newtheorem{algo}[algorithm]{Algorithm}

\newcommand\bA{\boldsymbol{A}}

\newcommand\bu{{\boldsymbol{u}}}
\newcommand\bul{\boldsymbol{u}^{\ell}}

\newcommand\tbu{\widetilde{\boldsymbol{u}}}

\newcommand\bee{\mathbf{e}}
\newcommand\beel{\mathbf{e}^\ell}
\newcommand\e{\boldsymbol{\epsilon}}
\newcommand\el{\boldsymbol{\epsilon}^\ell}
\newcommand\eell{\boldsymbol{\epsilon}^{\ell-1}}
\newcommand\te{\boldsymbol{\tilde{\epsilon}}}
\newcommand\tel{\boldsymbol{\tilde{\epsilon}}^\ell}

\newcommand\buo{\boldsymbol{u}_0}

\newcommand\bv{{\boldsymbol{v}}}
\newcommand\bB{{\boldsymbol{B}}}

\newcommand\bue{\boldsymbol{u}^\varepsilon}

\newcommand\buel{\boldsymbol{u}^{{\varepsilon},{\ell}}}
\newcommand\tbuel{\boldsymbol{\tilde{u}}^{{\varepsilon},{\ell}}}
\newcommand\tbuell{\boldsymbol{\tilde{u}}^{{\varepsilon},{\ell-1}}}
\newcommand\buem{\boldsymbol{u}^{{\varepsilon},{m}}}

\newcommand\buull{\boldsymbol{u}^{\ell-1}}

\newcommand\tbul{\boldsymbol{\tilde{u}}^{\ell}}
\newcommand\buell{\boldsymbol{u}^{\varepsilon,\ell-1}}

\newcommand\bx{\boldsymbol{x}}

\newcommand\bz{{\boldsymbol{z}}}

\newcommand\tbzel{{\boldsymbol{\tilde{z}}}^{\varepsilon,\ell}}

\newcommand\bze{\boldsymbol{z}^\varepsilon}
\newcommand\bzel{\boldsymbol{z}^{\varepsilon,\ell}}
\newcommand\bzell{\boldsymbol{z}^{\varepsilon,\ell-1}}

\newcommand\bd{\boldsymbol{d}}
\newcommand\bP{\boldsymbol{P}}

\newcommand\bw{\boldsymbol{w}}

\newcommand\bW{\boldsymbol{W}}

\newcommand\dv{\mathrm{div}\,}

\newcommand\per{\mathrm{per}}

\newcommand\rhoe{{\rho^\varepsilon}}

\newcommand\bb{\boldsymbol{b}}
\newcommand\bQ{{\boldsymbol{Q}}}
\newcommand\bfi{\boldsymbol{\varphi}}
\newcommand\bve{{\boldsymbol{v}^\varepsilon}}
\newcommand\bvel{{\boldsymbol{v}^{\varepsilon,\ell}}}
\newcommand\bvell{{\boldsymbol{v}^{\varepsilon,\ell-1}}}
\newcommand\tbvel{{\boldsymbol{\tilde{v}}^{\varepsilon,\ell}}}

\newcommand\psiel{{\psi}^{\varepsilon,\ell}}
\newcommand\psiell{{\psi}^{\varepsilon,\ell-1}}
\newcommand\roel{{\rho}^{\varepsilon,\ell}}

\newcommand\troel{{\tilde{\rho}}^{\varepsilon,\ell}}

\newcommand\vrol{{\varrho}^{\ell}}

\newcommand\tvrol{{\tilde{\varrho}}^{\ell}}

\newcommand\mC{\mathcal{C}}
\newcommand\cD{\mathscr{D}}
\newcommand\cF{\mathscr{F}}
\newcommand\HH{\mathbb{H}}
\newcommand\LL{\mathbb{L}}
\newcommand{\cL}{\mathscr{L}}
\newcommand\NN{\mathbb{N}}
\newcommand\PP{\mathbb{P}}
\newcommand\RR{\mathbb{R}}
\newcommand\EE{\mathbb{E}}
\newcommand\VV{\mathbb{V}}
\newcommand\WW{\mathbb{W}}

\newcommand\FF{\mathbb{F}}

\newcommand\mE{\mathcal{E}}
\newcommand\tmE{\mathcal{\widetilde{E}}}
\newcommand\mF{\mathcal{F}}
\newcommand\mK{{\boldsymbol{\mathcal{K}}}}
\newcommand\mH{{\boldsymbol{\mathcal{H}}}}

\newcommand{\I}{\texttt{I}}
\newcommand{\II}{\texttt{II}}
\newcommand{\III}{\texttt{III}}
\newcommand{\IV}{\texttt{IV}}

\newcommand{\tI}{\widetilde{\texttt{I}}}
\newcommand{\tII}{\widetilde{\texttt{II}}}
\newcommand{\tIII}{\widetilde{\texttt{III}}}
\newcommand{\tIV}{\widetilde{\texttt{IV}}}
\newcommand{\tV}{\widetilde{\texttt{V}}}

\newcommand{\pre}{\mathsf{p}}
\newcommand{\preel}{\mathsf{p}^{\varepsilon,\ell}}
\newcommand{\tpreel}{\tilde{\mathsf{p}}^{\varepsilon,\ell}}

\newcommand{\tprel}{\widetilde{\mathsf{p}}^\ell}
\newcommand{\qre}{\mathsf{q}}

\newcommand{\prel}{\mathsf{p}^\ell}
\newcommand{\qrel}{\mathsf{q}^\ell}
\newcommand{\pree}{\mathsf{p}^\varepsilon}

\newcommand{\pie}{{\pi^\varepsilon}}
\newcommand{\piel}{{\pi^{\varepsilon,\ell}}}
\newcommand{\tpiel}{{\tilde{\pi}^{\varepsilon,\ell}}}

\newcommand{\vpil}{{\varpi}^\ell}
\newcommand{\vpiel}{{\varpi}^{\varepsilon,\ell}}
\newcommand{\xiel}{{\xi}^{\varepsilon, \ell}}
\newcommand{\xiell}{{\xi}^{\varepsilon,\ell-1}}
\newcommand{\xiem}{{\xi}^{\varepsilon,m}}
\newcommand{\fil}{{\phi}^{\ell}}
\newcommand{\fiill}{{\phi}^{\ell-1}}
\newcommand{\fiel}{{\phi}^{\varepsilon,\ell}}
\newcommand{\fiell}{{\phi}^{\varepsilon,\ell-1}}
\newcommand{\fiem}{{\phi}^{\varepsilon,m}}
\newcommand{\vro}{{\varrho}}
\newcommand{\bsig}{\boldsymbol{\sigma}}
\newcommand{\bsigl}{\boldsymbol{\sigma}^\ell}
\newcommand{\bsigll}{\boldsymbol{\sigma}^{\ell-1}}
\newcommand{\tbsigl}{\boldsymbol{\tilde{\sigma}}^\ell}

\newcommand{\trace}{\mathrm{Tr}}
\newcommand\la{\left\langle }
\newcommand\ra{\right\rangle }
\newcommand\tB{\tilde{B}}

\newcommand\tC{\widetilde{C}}
\newcommand\tbB{\tilde{\boldsymbol{B}}}
\newcommand\tb{\tilde{b}}

\newcommand{\tOm}{\widetilde{\Omega}}

\newcommand\dun{{\delta_{1}}}
\newcommand\ddeu{{\delta_{2}}}

\newcommand\tNLT{\widetilde{NLT}}
\newcommand\sA{\mathscr{A}}

\newcommand{\sel}{\textsc{e}^{\ell}}
\newcommand{\tsel}{\tilde{\textsc{e}}^{\ell}}

\newcommand{\sql}{\textsc{q}^{\ell}}
\newcommand{\sell}{\textsc{e}^{\ell-1}}

\newcommand{\noise}{\mathrm{Noise}}
\newcommand{\tbsig}{ \boldsymbol{\tilde{\sigma}}}
\newcounter{lil}

\newcounter{gr111}
\newenvironment{steps}
{\begin{list} {\bf Step (\Roman{gr111})$\, $} {\usecounter{gr111}
			\setlength{\labelwidth}{-0.2cm} \setlength{\leftmargin}{0.0cm}
			\setlength{\topsep}{0.1cm} \setlength{\itemsep}{0.2cm}
			\setlength{\parsep}{0.1cm} \setlength{\itemindent}{0.4cm}
			\setlength{\parskip}{0.0cm}}} {\end{list}}
\author{Tsiry Randrianasolo}
\author{Erika Hausenblas}

\address[Erika Hausenblas]{Lehrstuhl Angewandte Mathematik,
	Montanuniversit\"at Leoben,
	Peter-Tunner-Stra{\ss}e 25-27\\
	AT-8700 Leoben}

\address[Tsiry Randrianasolo]{Fakult\"at f\"ur Mathematik,
	Universit\"at Bielefeld,
	Universit\"atsstra\ss e 25\\
	DE-33615 Bielefeld}
\email[Corresponding author]{\href{mailto:trandria@math.uni-bielefeld.de}{trandria@math.uni-bielefeld.de}}

\title[Time-discretization of stochastic 2-D Navier--Stokes equations]{Time-discretization of stochastic 2-D Navier--Stokes equations with a penalty-projection method}
\date{\today}
\thanks{This research was supported by the Austrian Science Fund (FWF): P26958 and the German Research Council as part of the Collaborative Research Center SFB 1283.}
\begin{document}
	\begin{abstract}
		A time-discretization of the stochastic incompressible Navier--Stokes problem by penalty method is analyzed. Some error estimates are derived, combined, and eventually arrive at a speed of convergence in probability of order 1/4 of the main algorithm for the pair of variables velocity and pressure. Also, using the law of total probability, we obtain the strong convergence of the scheme for both variables.
		%
		%
	\end{abstract}
	\maketitle
	\section{Introduction}\label{sec:introduction}
	Let $T>0$ and $\mathfrak{P}:=(\Omega, \mF,\FF,\PP)$
	be a filtered probability space with the filtration $\FF:=(\cF_t)_{0\leq t\leq T}$ satisfying the usual conditions. We refer to the following system of equations as the stochastic \textit{incompressible Navier--Stokes problem} (SNS),
	\begin{equation}
	\label{eq:NSE1}\left\{
	\begin{split}
	\bu_t - \nu\Delta\bu+[\bu\cdot\nabla]\bu +  \nabla \pre
	&=  \dot{\bW}, \mbox{ in } \RR^2,
	\\
	\dv\bu
	&= 0, \mbox{ in } \RR^2.
	\end{split}\right.
	\end{equation}
	Here $\bu = \left\{ \bu(t,\bx):t\in[0,T]\right\}$ and $\pre=\left\{\pre(t,\bx):t\in[0,T]\right\}$ are unknown stochastic processes on $\RR^2$, representing respectively the velocity and the pressure of a fluid with kinematic viscosity $\nu$ filling the whole space $\RR^2$, in each point of $\RR^2$.
	
	In $\RR^2$, we endow \eqref{eq:NSE1} with an initial condition,
	\begin{align*}
	\bu(0,\bx) &= \bu_0(\bx)
	\intertext{and periodic boundary conditions, }
	\bu(t,\bx+ L\bb_j) &= \bu(t, \bx), \qquad j=1,2,\quad t\in [0,T] ,
	\end{align*}
	where $\bu$ has a vanishing spatial average. Here $(\bb_1,\bb_2)$ is the canonical basis of $\RR^2$ and $L>0$ is the period in the $j$th direction; $D = (0,L)\times(0,L)$ is the square of the period. The term $\bW:= \{\bW(t): t\in [0,T]\}$ is a $\mK$-valued Wiener process where $\mK$ is a separable Hilbert space.

	An incompressible fluid flow is usually modeled with a deterministic Navier--Stokes equation. The stochastic Navier--Stokes~\eqref{eq:NSE1} is a well known model that captures fluid instabilities under ambient noise \cite{birnir2008turbulence} or small scales perturbation for homogeneous turbulent flow, see e.g. \cite{bernard2002turbulent}, \cite{birnir2013the}, and \cite{pope2000turbulent}.
	
	Strong approximation of Stochastic Partial Differential Equations (SPDEs), such as \cref{eq:NSE1}, is mostly the natural approach because of its link with the numerical analysis of deterministic equations. However, this type of approximation is often inaccessible for nonlinear SPDEs. Indeed, when the nonlinearity is neither globally Lipschitz nor monotone, weak convergence or convergence in probability are frequently considered, see e.g. \cite{banas2014a}, \cite{brzezniak2013finite}, \cite{debouard2004a}, \cite{debussche2006convergence},  \cite{flandoli1995time}, \cite{hou2006wiener}, and \cite{milstein20161layer}. Another notion, the \textit{speed of convergence in probability}, was first put forward by Printems in \cite{printemps2001on} for some parabolic SPDEs. Regardless of the type of convergence, we may also have to consider different approaches according to the characteristic of the equation. In particular for the SNS, we can use e.g. a numerical approximation using an Ornstein--Uhlenbeck as an auxiliary step such as in \cite{flandoli1995time}, or using splitting methods such as in \cite{bessaih2014splitting,erich2012time}, or using the Wiener chaos expansion such as in \cite{hou2006wiener}, or using the layer method (probabilistic representation) such as in \cite{milstein20161layer}. Carelli and Prohl proved in \cite{carelli2012rates} that a speed of convergence in probability can be derived from some direct numerical approximations of the SNS. Here the convergence concerns only one variable, the velocity field.

	The SNS shares the same complexity as its deterministic counter part, when it comes to computations. Velocity and pressure are both coupled by the incompressibility constraint, which often requires a saddle point problem to solve. To break this saddle point character of the system, velocity and pressure are decoupled by perturbing the divergence free condition by a penalty method \cite[Chapter 3]{lions1969quelques} and choosing a penalty operator in a similar fashion as in \cite{courant1943variational}. This consists, for every $\varepsilon>0$, to solve the penalized version of \eqref{eq:NSE1}, i.e.
	\begin{equation}
	\label{eq:penalized NSE1}\left\{
	\begin{split}
	\bue_t - \nu\Delta\bue + [\bue\cdot\nabla]\bue+ \frac12(\dv\bue)\bue +  \nabla \pree &= \dot{\bW}, \mbox{ in } \RR^2,
	\\
	\dv\bue+\varepsilon\pree &= 0, \mbox{ in } \RR^2.
	\end{split}\right.
	\end{equation}	
	This belongs to a more general class of approximation methods for the Navier--Stokes equation, called projection and quasi-compressible methods. This includes the artificial compressibility method, the pressure stabilization, and the pressure correction method. For a complete survey or review on these methods, the reader is referred for instance to \cite{guermond2006an} or the monograph \cite{prohl2013projection}. Even though these methods are already very popular and efficient in the deterministic framework, the paper of Carelli, Hausenblas, and Prohl, see \cite{erich2012time}, is the only work, which treats on projection and quasi-compressible methods for the stochastic Stokes equation by using the pressure stabilization and the pressure correction methods to derive an algorithm based on a time marching strategy. The artificial compressibility method has already been used to prove existence and pathwise uniqueness of global strong solutions of SNS, see \cite{manna2008stochastic}, or adapted solutions to the backward SNS by a local monotonicity argument, see \cite{yin2010stochastic}. Concerning the penalty method, it has been introduced in \cite{temam1968une} by Temam for the deterministic Navier--Stokes equations where he established its convergence. Since then, the method has been improved by Shen with the addition of error estimates in a sequence  of papers including \cite{shen1992onnumerische} and \cite{shen1995on}. It has been used (with a different penalty operator) in a stochastic framework in \cite{capinski2001on} as an auxiliary step to prove the existence of a spatially homogeneous solution of a SNS driven by a spatially homogeneous Wiener random field.
	
	In this paper, we study a semi-implicit time-discretization scheme for the full stochastic incompressible 2D Navier--Stokes equation based on the penalized system \cref{eq:penalized NSE1}. Formally, the scheme consists 
	of solving the following equations:
	\\[7pt]
	Given $0<\eta<1/2$, $\alpha>1$, $\buo$, $\phi^0= 0$. For $\ell = 1,\ldots,M$:
	\begin{align*}
	\intertext{$\bullet$ Step 1 (Penalization): Find $\tbul$ such that}
	&\tbul -\nu k\Delta\tbul+k\tB(\tbul, \tbul)-k^{1-\eta}\nabla \dv\tbul= \Delta_\ell\bW+\buull-k\nabla\fiill;
	\intertext{$\bullet$ Step 2: Find $\fil$ such that}
	&\Delta\fil= \Delta\fiill +(\alpha k)^{-1}\dv\tbul;
	\end{align*}
	{$\bullet$ Step 3 (Projection): $\bul = \bP_\HH\tbul$}, i.e.
	\begin{equation*}
	\begin{split}
	\bul =\tbul- \alpha k \nabla(\fil - \fiill),\qquad
	\prel=\tprel + \fil + \alpha(\fil - \fiill).
	\end{split}
	\end{equation*}
	More details are given in \cref{algo:penalty method}. 
	We focus on the time-discretization, since different technical endeavors may obscure the main difficulty of the time-discretization. A paper which is similar to ours is \cite{carelli2012rates}, where the authors show the convergence in probability of a space-time discretization of stochastic incompressible Navier--Stokes in 2D. The numerical schemes they use are implicit/semi-implicit in time and use a divergence-free finite element pairing such as the Scott--Vogelius finite element for the velocity and the pressure. The proof needs also some a priori estimates of the approximate solution in $\VV$, the divergence-free space with finite enstrophy. These estimates are obtained by means of the additional orthogonal property of the nonlinear term in 2D and under periodic boundary conditions, i.e. $\langle[\bu.\nabla]\bu,\Delta\bu\rangle = 0$ for each $\bu\in \VV$. As we see in \cref{eq:penalized NSE1}, the approximate solution is only slightly compressible, thus $\bue\not\in\VV$. Even with the projection step added, the additional orthogonal property required in \cite{carelli2012rates} is inapplicable here. To overcome this issue we use the classical decomposition of the SNS into an Ornstein--Uhlenbeck process and a deterministic SNS. This decomposition has already been used for different purpose, e.g. in \cite{brzezniak2016a}, \cite{fernando2015mild}, \cite{flandoli1994dissipativity}, \cite{flandoli1999weak}, and \cite{flandoli1995time}. The algorithm depends on the spatial perturbation parameter $\varepsilon>0$, a stability preserving parameter $\alpha>1$, and the time-step $k$. If we fix $\varepsilon = k^\eta$ with some $0<\eta<1/2$ and with any $\alpha>1$, a  speed of convergence in probability of order 1/4 is obtained for both velocity and pressure. Then, by means of the law of total probability, we deduce strong convergence of the scheme for both variables velocity and pressure. In this context, we respond to the lack of results regarding (speed of) convergence for the pressure iterates from algorithms based on pseudo-compressible and projection method for stochastic (Navier)--Stokes equations addressed by \cite{erich2012time}.
	
	This paper is organized as follows. In \cref{se:preliminaries}, we introduce the assumptions and notations used and review some of the basic
	facts of the SNS, which are important for the proof,  such as the time regularity of the solution and present a splitting argument that will be used later on. In the \cref{sec:auxiliary results}, we develop stability of the main algorithm and derive error estimates for some auxiliary algorithms. In \cref{sec:main results}, we treat the speed of convergence in probability, then the strong convergence of the main algorithm.
	
	\section{Preliminaries}
	\label{se:preliminaries}
	In this section, we present the assumptions and notations used in this work. We also prove the time regularity of the pressure.
	As a preparatory work, before going into the numerical analysis, we formulate \eqref{eq:NSE1} according to the classical decomposition of the SNS into an Ornstein--Uhlenbeck process and a deterministic Navier--Stokes depending on a stochastic process.
	\subsection{Functional settings and notations}
	\label{subsec:functional settings and notations}	
	To introduce a spatial variable process, i.e.\ a vector-valued process to the Brownian motion $\bW$, we introduce a family  of mutually independent and identically distributed real-valued Brownian motions $\{\beta_j(t): t\in[0,T]\}$, $j\in\NN$, and a covariance $\bQ$. If $\bQ\in\cL(\mK)$ (the space of bounded linear operators from $\mK$ to $\mK$) is non-negative definite and symmetric with an orthonormal basis $\{\bd_j: j\in \mathbb{N} \}$ of eigenfunctions with corresponding eigenvalues $ q_j\geq 0$ such that $\sum_{j\in\NN}q_j<\infty$, then $\bQ\in\cL_1(\mK)$ (the space of trace-class operator on $\mK$) and the series
	\[
	\bW(t)= \sum_{j=1}^\infty \sqrt{q_j} \beta_j(t) \bd_j, \qquad \forall\; t\in [0,T],
	\]
	converges in $L^2(\Omega; \mC([0,T]; \mK))$ and it defines a $\mK$-valued Wiener process with covariance operator $\bQ$ also called \textit{$\bQ$-Wiener process}. Furthermore,  for any $\ell\in\NN$ there exists a constant $C_\ell >0$ such that
	\begin{equation}
	\label{eq:increment of a wiener process}
	\EE \lVert \bW(t)- \bW(s)\rVert^{2\ell}_\mK \leq C_\ell ( t-s)^\ell \left(\trace \bQ\right)^\ell,\qquad \forall\;t\in [0,T]\quad\mbox{and}\quad \forall\;s\in [0,t).
	\end{equation}
	
	Let $\mH$ be another separable Hilbert space. We define by $\cL_2(\mK_\bQ,\mH)$ the space of Hilbert--Schmidt operator from $\mK_\bQ$ to $\mH$, where $\mK_\bQ$ is the separable Hilbert space defined by $\mK_\bQ:=\bQ^{1/2}\mK$.
	
	We can define the $\mH$-valued It\^o integral with respect to a $\bQ$-Wiener process $\bW$ by
	\[
	\int_0^t \Phi(s)d\bW(s) :=\sum_{j=1}^{\infty}\int_0^t\Phi(s)\sqrt{q_j} \bd_jd\beta_j(s) ,\quad\forall\; t\in [0,T]
	\]
	which is also a $\mH$-valued martingale satisfying the Burkholder--Davis--Gundy inequality (see \cite[Theorem 3.3.28]{karatzas2014brownian}), given by
	\begin{equation}
	\label{eq:BHG inequality}
	\EE\sup_{0\le s\le t}\bigg\Vert \int_0^s \Phi(\tau)d\bW(\tau)\bigg\Vert_{\mH}^{2r} \leq C_r\biggl(\int_0^t \lVert \Phi(\tau) \lVert^2_{\cL_2(\mK_Q, \mH)} d\tau \biggr)^r,\;\forall\; t\in [0,T],\;\forall\; r>0.
	\end{equation}
	
	In the case of scalar functions, we denote the usual Sobolev spaces by $W^{m,2}(D)$ $(m=0,1,2,\ldots,\infty)$. The corresponding scalar product and the corresponding norm for any nonnegative integer $m$ is denoted by
	\[
	(\bu,\bv)_m = \int_{D}\sum_{\ell=0}^{m} \partial^\ell\bu\, \partial^\ell\bv\, d\bx \quad\mbox{ and }\quad
	\lVert \bu\rVert_m = \lVert \bu\rVert_{W^{m,2}} = (\bu,\bu)_m^{1/2}.
	\]
	By $W^{m,2}_0(D)$, we denote the closure in $W^{m,2}(D)$ of the space $\mC^\infty_0(D)$ of all smooth functions defined on $D$ with compact support. Further, $W^{-m,2}(D)$ is the space that is dual to $W^{m,2}(D)\cap W^{1,2}_0(D)$.
	Particularly for $m=0$, the space $W^{m,2}(D)$ is usually denoted by $L^2(D)$ and then the scalar product and norm are denoted simply by $(\cdot,\cdot)$ and $\| \cdot\|$, respectively. We reserve the notation $\la\cdot,\cdot \ra$ for the duality bracket. In general, we denote the usual Lebesgue spaces by $L^p$, $1\leq p\leq \infty $, which are endowed with the standard norms denoted by $\|\cdot\|_{L^p}$. We denote by $L^p_\per$ and $W^{m,2}_\per$ the Lebesgue and Sobolev spaces of
	functions that are periodic and have vanishing spatial average, respectively. The spaces of vector-valued functions will be indicated with Blackboard bold letters, for instance ${\LL^2_\per}:= (L^2_\per)^2$. In further analyses, we will not distinguish between the notation of inner products and norms in scalar or vector-valued applications.
	
	The two spaces frequently used in the theory of Navier--Stokes equations are
	\begin{align*}
	\HH = \left\lbrace \bv\in \LL^2_\per(D):\dv \bv = 0 \mbox{ in } \RR^2\right\rbrace \quad\mbox{ and }\quad
	\VV = \left\lbrace \bv\in {\WW}^{1,2}_\per(D):\dv \bv= 0\mbox{ in } \RR^2\right\rbrace.
	\end{align*}
	The space $\VV$ is a Hilbert space with the scalar product $(\cdot,\cdot)_1$ and the Hilbert norm induced by $\WW^{1,2}$.
	
	Let $\bP_{\HH}$ denote the $\LL^2$-projection on the space $\HH$ also known as \textit{Helmholtz--Leray projector}. As an orthogonal projection, it satisfies the following identity
	\begin{equation}\label{eq:pty ortho proj}
	\langle \bP_{\HH}\bv - \bv,\bP_{\HH}\bv\rangle = 0,\quad\forall\,\bv\in\LL^2_\per.
	\end{equation}
	The projection $\PP_{\HH}$ is continuous from $\WW^{1,2}(D)_0$ into $\WW^{1,2}(D)$ (cf. \cite[Remark 1.6]{temam1984navier} and \cite[Proposition IV.3.7.]{boyer2012mathematical}) and we can find a positive constant $C = C(D)$ such that
	\begin{equation}
	\label{eq:stability of leray projection}
	\lVert\bP_\HH\bu \rVert_1\leq C\lVert\bu \rVert_1,\qquad \forall\;\bu\in\WW^{1,2}(D).
	\end{equation}
	Due to the Helmholtz--Hodge--Leray decomposition, any function $\bu\in \LL^2(D)$ can be represented as
	$
	\bu = \bP_{\HH}\bu + \nabla \qre,
	$
	where $\qre$ is a scalar $D$-periodic function such that $\qre\in L^2_\per(D)$. It is natural to introduce the notation $\bP_{\HH}^{\perp}\bu:=\nabla\qre$ and hence write
	\[
	\bu = \bP_{\HH}\bu + \bP_{\HH}^{\perp}\bu,
	\quad\mbox{ with }\quad
	\bP_{\HH}^{\perp}\bu\in\HH^{\perp}=\left\{ \bv:\bv\in\LL^2(D),\,\bv=\nabla\qre\right\}.
	\]
	%
	With periodic boundary conditions the Stokes operator $\bA=-\bP_\HH\Delta$ coincides with the Laplacian operator $-\Delta$. The operator $\bA$ can be seen as an unbounded positive linear selfadjoint operator on $\HH$ with domain $\cD(A) = {\WW}^{2,2}\cap\VV$. We can define the powers $\bA^\alpha$, $\alpha\in\RR$, with domain $\cD(\bA^\alpha)$. The norm $\lVert \bA^{s/2}\bu\rVert$ on $\cD(\bA^{s/2})$ is equivalence to the norm induced by $\WW^{s,2}_0(D)$. In addition, we also have the following equivalence of norm:
	
	\begin{prop}[Equivalence of norms]\label{prop:equivalence of norms} There exist positive numbers $c_1$ and $c_2$ such that $\forall\, \bu\in \HH$:
		\begin{itemize}
			\item[$(i)$] $\lVert \bA^{-1}\bu\rVert_s\leq c_1\lVert\bu \rVert_{s-2},\quad s=1,2;$
			\vspace{5pt}
			\item[$(ii)$]$c_2\lVert \bu\rVert^2_{-1}\leq (\bA^{-1}\bu,\bu)\leq c_1^2\lVert\bu \rVert_{-1}^2$.
		\end{itemize}
	\end{prop}
	
	
	\begin{proof}
		The reader is referred to \cite[Equation (2.1)]{shen1992onsiam} or \cite[Lemma 2.3]{prohl2013projection} for the proof. It relies on the elliptic regularity of the Stokes operator and the definition of negative Sobolev norms.
	\end{proof}
	
	We now introduce some operators usually associated with the Navier--Stokes equations and their approximations. In particular,
	\begin{align*}
	&\bB(\bu,\bv) = [\bu\cdot\nabla]\bv,\quad \tbB(\bu,\bv) = \bB(\bu,\bv) +(\dv\bu)\bv/2,\\
	&b(\bu,\bv,\bw) = \la\bB(\bu,\bv), \bw\ra,\quad \tb(\bu,\bv,\bw) = \langle\tbB(\bu,\bv),\bw\rangle.
	\end{align*}
	The trilinear forms $b$ and $\tb$ satisfy the following properties:
	
	\textit{Skew-symmetry property}
	\begin{equation}
	\label{eq:skew-symmetry}
	\begin{split}
	b(\bu,\bv,\bw) &= -b(\bu,\bw,\bv),\qquad \bu\in\HH\;\mbox{ and }\;\bv,\bw \in \VV,\\
	\tb(\bu,\bv,\bw) &= -\tb(\bu,\bw,\bv),\qquad \bu,\bv\in\WW^{1,2}(D)\;\mbox{ and }\;\bw \in \WW^{1,2}_\per(D).
	\end{split}
	\end{equation}
	
	\textit{Orthogonal property}
	\begin{equation}
	\label{eq:orthogonal property 1}
	\begin{split}
	b(\bu,\bv,\bv)=0,\quad\forall\;\bu\in\HH,\,\forall\;\bv\in\WW^{1,2}_\per(D);\quad
	\tb(\bu,\bv,\bv)=0,\quad\forall\;\bu,\bv\in\WW^{1,2}_\per(D).
	\end{split}
	\end{equation}
	%
	
	The following estimates of the trilinear form $\tb$ will be used repeatedly in the upcoming sections. Let $\bv\in{\WW}^{2,2}(D)\cap{\WW}^{1,2}_\per(D)$ and $\bu,\bw\in {\WW}^{1,2}_\per(D)$; a combination of integration by parts and H\"older inequality gives
	\begin{align}
	\label{eq:estimate 1 of tb}
	\tb(\bu,\bv,\bw)&\leq \lVert \bu\rVert_{\LL^4}\lVert \bv\rVert_{1}\lVert \bw\rVert_{{\LL^4}}.
	\intertext{	From this estimate we can deduce using the Sobolev embedding $\WW^{1,2}(D)\subset\LL^4(D)$,}
	\label{eq:estimate 2 of tb}
	\tb(\bu,\bv,\bw)&\leq C(L)\lVert \bu\rVert_1\lVert \bv\rVert_{1}\lVert \bw\rVert_1,\qquad
	\intertext{or using the Ladyzhenskaya's inequality $\lVert \bu\rVert_{\LL^4}\leq C(L)\lVert \bu\rVert^{1/2}\lVert \bu\rVert_1^{1/2}$,}
	\label{eq:estimate 3 of tb}
	\tb(\bu,\bv,\bw)&\leq C(L)\lVert \bu\rVert^{1/2}\lVert \bu\rVert_1^{1/2}\lVert \bv\rVert_{1}\lVert \bw\rVert^{1/2}\lVert \bw\rVert_1^{1/2}.
	\end{align}
	To find more about the above properties or additional properties of $b$ or $\tb$, and other estimates, the reader is referred to \cite[Section 2.3]{temam1983navier}.

	\subsection{General assumption and spatial regularity of the solution}\label{subsec:gnl assump and spatial regularity}
	In the following we choose  $\mH=\VV$, i.e.~a solenoidal noise in SNS. An example of solenoidal noise is given in \cite[Section 6]{erich2012time}.
	We summarize the assumptions needed for data $\bW$, $\bQ$, and $\buo$:
	
	\crefname{enumi}{Assumption}{Assumptions}
	\begin{enumerate}[label=($S_\arabic*$)]
		\item\label{S1} For $\bQ\in\cL(\mK)$, let $\bW = \{\bW(t): t\in[0,T]\}$ be a $\bQ$-Wiener process with values in a separable Hilbert space $\mK$ defined on the stochastic basis $\mathfrak{P}$.
		\item\label{S4}  $\buo\in \VV$.
	\end{enumerate}
	
	\medskip
	
	In addition, we recall the notion of a strong solution to \eqref{eq:NSE1}.
	
	\begin{defn}[Strong solution]\label{def:strong solution} Let $T>0$ be given and let \cref{S1,S4} be valid, with $\mH=\VV$. A $\VV$-valued process $\bu = \{\bu(t,\cdot):t\in[0,T]\}$ on $(\cF_t)_{0\leq t\leq T}$ is a strong solution to \eqref{eq:NSE1} if
		\begin{enumerate}[label=(\roman*)]
			\item $\bu(\cdot,\cdot,\omega)\in \mC([0,T];\VV)\cap L^2(0,T;{\WW}^{2,2}\cap\VV)$\,$\PP\mbox{-a.s.}$,
			\item for every $t\in[0,T]$ and every $\bfi\in\VV$, there holds $\PP\mbox{-a.s.}$
			\[
			\begin{split}
			(\bu(t),\bfi) &+ \nu\int_0^t(\nabla\bu(s),\nabla\bfi)+ b(\bu(s),\bu(s),\bfi) ds = (\buo,\bfi)+\int_0^t\left(\bfi,d\bW(s)\right).\quad
			\end{split}
			\]
		\end{enumerate}
	\end{defn}
	%
	If \cref{S1} holds and $\mH = \VV$, we can prove (cf. \cite[Appendix 1]{flandoli1995martingale}) that the solutions $\bu$ of \eqref{eq:NSE1} as defined by Definition~\ref{def:strong solution} satisfies for $2\leq p<\infty$ the estimate
	\begin{align}
	\label{eq:L2-estimate of u}
	\EE\sup_{0\leq t\leq T}\lVert\bu(t) \rVert^p +\nu\EE \left[\int_{0}^{T}\lVert \bu(s)\rVert^{p-2}\lVert \nabla\bu(s)\rVert^2 ds\right]&\leq C_{T,p},
	\intertext{where $C_{T,p}= C_{T,p}(\trace \bQ,\EE\lVert
		\buo\rVert^p,\EE\lVert \buo\rVert^p_\VV)>0$. In addition to the above estimate, if \cref{S4} holds for $2\leq p<\infty$, it is proven in \cite[Lemma 2.1]{carelli2012rates} that $\bu$ satisfies also the estimates}
	\label{eq:V-estimate of u 1}
	\sup_{0\leq t \leq T}\EE \lVert \bu(t)\rVert^p_{\VV} +  \nu \EE\left[\int_0^T\lVert \bu(s)\rVert_\VV^{p-2}\lVert\bA\bu(s) \rVert ^2 ds\right]&\leq C_{T,p},
	\\
	\label{eq:V-estimate of u 2}
	\mbox{and}\quad \EE \sup_{0\leq t\leq T}\lVert \bu(t)\rVert^p_{\VV}&\leq C_{T,p}.
	\end{align}
	
	We associate a pressure $\pre$ to the velocity $\bu$ by using a generalization of the de Rham theorem to processes, see \cite[Theorem 4.1]{langa2003existence}.
	In addition, we also have the following estimate for the pressure:
	
	\begin{prop}\label{prop:bound for the pressure}
		Under \cref{S1,S4}, there exists a  constant $C>0$ such that the velocity fields $\bu$ and pressure fields $\pre$ satisfy $\PP$-a.s.
		\begin{equation}
		\label{eq:bound for the pressure}
		\lVert\pre(t) \rVert\leq C\lVert \bA^{1/2}\bu(t)\rVert^2,\qquad\forall\; t\in[0,T].
		\end{equation}
	\end{prop}
	\begin{proof}
		To show the Proposition \ref{prop:bound for the pressure} we project equation  \eqref{eq:NSE1} into $\HH^\perp$ using the projection operator $\bP^{\perp}_\HH$. Since $\bP^{\perp}_\HH$ commutes with the Laplacian operator (we work with a periodic boundary condition) and $\dv\bu = 0$, then
		\[
		\bP^{\perp}_\HH\bu_t = 0\qquad\mbox{ and}\qquad \bP^{\perp}_\HH\Delta\bu = 0.
		\]
		In \cref{S1} we suppose that the forcing term is divergence-free, hence, each solenoidale term vanishes after projection with $\bP^{\perp}_\HH$. The remaining terms give
		\begin{equation*}
		\label{eq:grad pressure wrt velocity 1}
		\nabla\pre (t) = -\bP_\HH^{\perp}\bB(\bu(t),\bu(t)),\qquad\forall\;t\in[0,T].
		\end{equation*}
		It follows from \cite[Lemma~2.2]{giga1985solution} for $r= 2,n=2,\,\delta = 1/2,\theta=\rho=1/2$ and \eqref{eq:estimate 3 of tb}
		that
		\begin{align*}
		\lVert\nabla\pre(t)\rVert_{{-1}}= \lVert\bP_\HH^{\perp}\bB(\bu(t),\bu(t))\rVert_{{-1}}
		&\leq \lVert\bP_\HH\bB(\bu(t),\bu(t))\rVert_{{-1}} + \lVert\bB(\bu(t),\bu(t))\rVert_{{-1}}
		\\
		&\leq C\lVert \bA^{1/2}\bu(t)\rVert^2.
		\end{align*}
		Finally, it follows by the  Ne\v cas inequality for functions with vanishing spatial average (cf. \cite[Proposition IV.1.2.]{boyer2012mathematical}), that there exists a constant $C > 0$, such that
		\begin{equation}
		\label{eq:first estimate of the pressure}
		\lVert\pre(t)\rVert \leq C\lVert \bA^{1/2}\bu(t)\rVert^2,\qquad\forall\;t\in[0,T].
		\end{equation}
		The constant $C$ comes from the  Ne\v cas inequality, more precisely from the definition of the norm in $\WW^{-1,2}$ by the Fourier transform. Therefore,  $C$ depends on the spatial dimension $d$ and the $L^p$-estimates for the Fourier transform multipliers. Here we have $d=2$ and $p=2$, but a similar estimate can be obtained for $d\geq 2$ and $2\leq p<\infty$, see \cite[Corollaries 1 and 2]{cattabriga1961su} and \cite[Lemma 7.1]{necas2011direct}.
	\end{proof}
	\subsection{Regularity in time of the solution of the SNS}

	\begin{lem}\label{lem:holder continuity of u}
		Suppose that \cref{S1} holds, and $\mH = \VV$. For the solution of \eqref{eq:NSE1}, with $\buo\in  \VV$, $2\leq p<\infty,$ we can find a constant $C=C(T,p,L)>0$, such that for $0\leq s<t\leq T$ we have
		\begin{enumerate}[label=$(\roman*)$]
			\vspace{0.75em}
			\item\label{item:holder continuity of u 1}$\EE\lVert \bu(s)- \bu(t)\rVert^{p}_{\LL^4}\leq C\lvert s-t\rvert^{\eta p}\;\quad \forall\; 0<\eta<\tfrac12$,
			\vspace{0.65em}
			\item\label{item:holder continuity of u 2}$\EE\lVert \bu(s)- \bu(t)\rVert^{p}_{\VV}\;\,\leq C\lvert s-t\rvert^{\tfrac{\eta p}{2}} \quad \forall\; 0<\eta<\tfrac12$,
			\vspace{0.75em}
			\item\label{item:holder continuity of u 3}$\EE\lVert \pre(s)-\pre(t)\rVert^{\tfrac{p}{2}}\;\;\;\leq  C\lvert s-t\rvert^{\tfrac{\eta p}{4}}\quad \forall\; 0<\eta<\tfrac12$.
		\end{enumerate}
	\end{lem}
	
	\medskip
	
	\begin{proof} The assertions $(i)$ and $(ii)$ are direct quotations of \cite[Lemma 2.3]{carelli2012rates}. We only prove the assertion $(iii)$. Let $t\in[0,T]$. Applying the projection $\bP_\HH^{\perp}$ on \eqref{eq:NSE1} we get
		\[
		\nabla\pre(t) = -\PP_{\HH}^{\perp}\bB(\bu(t),\bu(t)).
		\]
		The following identity  holds 	for $0\leq s<t$
		\begin{align}
		\nabla(\pre(s)-\pre(t)) = \PP_{\HH}^{\perp}\bB(\bu(t),\bu(t)-\bu(s)) + \PP_{\HH}^{\perp}\bB(\bu(t)-\bu(s),\bu(s)).
		\end{align}
		Using the Ne\v cas inequality for vanishing spatial average  and Proposition~\ref{prop:bound for the pressure}, we obtain
		\begin{align*}
		\lVert\pre(s)-\pre(t)\rVert &\leq \lVert\PP_{\HH}^{\perp}\bB(\bu(t),\bu(t)-\bu(s))\rVert_{-1} + \lVert\PP_{\HH}^{\perp}\bB(\bu(t)-\bu(s),\bu(s))\rVert_{-1}
		\\
		&\leq C\lVert\bu(t)\rVert_1\lVert\bu(t)-\bu(s)\rVert_1 + C\lVert\bu(t)-\bu(s)\rVert_1\lVert\bu(s)\rVert_1.
		\intertext{Taking the $p/2$-moment and using the H\"older inequality we get}
		\EE\lVert\pre(s)-\pre(t)\rVert^{p/2}
		&
		\leq C(L)\left[\left(\EE\lVert\bu(t)\rVert_{1}^{p}\right)^{1/2}+\left(\EE\lVert\bu(s)\rVert_{1}^{p}\right)^{1/2}\right]\left(\EE\lVert\bu(t)-\bu(s)\rVert_1^{p}\right)^{1/2}.
		\end{align*}
		We deduce from \eqref{eq:V-estimate of u 1} and the assertion $(i)$ of the present lemma that
		\begin{equation}
		\EE\lVert\pre(s)-\pre(t)\rVert^{p/2}\leq C_{T,2}(L)\lvert s-t\rvert^{\eta p/4}.
		\end{equation}

	\end{proof}
	
	\subsection{Classical decomposition of the solution} Before going to the next section we introduce a splitting argument which is essential for the rest of the paper. We consider the auxiliary Stokes equation
	\begin{equation}
	\label{eq:NSE auxiliary}
	\left\{
	\begin{split}
	d\bz + \left[- \nu\Delta\bz +  \nabla \pi\right]dt &= d\bW, \mbox{ in } \RR^2,
	\\
	\dv\bz &= 0, \mbox{ in } \RR^2,
	\end{split}
	\right.
	\end{equation}
	with $\bz(0) = 0$ and which corresponds the penalized system
	\begin{equation}
	\label{eq:NSE auxiliary penalized}
	\left\{
	\begin{split}
	d\bze + \left[- \nu\Delta\bze +  \nabla \pie\right]dt &= d\bW, \mbox{ in } \RR^2,
	\\
	\dv\bze + \varepsilon\pie &= 0, \mbox{ in } \RR^2,
	\end{split}
	\right.
	\end{equation}
	with $\bze(0)= 0$.
	
	As already pointed out by \cite{carelli2012rates}, the nonlinear term of the SNS does not allow  to use a Gronwall argument. To tackle this issue, we use the classical decomposition
	of the solution $\bu$ into two parts: one part, given by the process  $\bz$,  will be random,  but linear;  the other part, denoted by  $\bv$,
	will be nonlinear, but deterministic. In this way, we  write the solution of \eqref{eq:NSE1} as $\bu = \bv + \bz$, where $\bv$ solves
	\begin{equation}
	\label{eq:NSE auxiliary 2}
	\left\{
	\begin{split}
	\frac{d\bv}{dt} + \tbB(\bv + \bz,\bv + \bz) - \nu\Delta\bv +  \nabla \rho &= 0, \mbox{ in } \RR^2,
	\\
	\dv\bv &= 0, \mbox{ in } \RR^2,
	\end{split}
	\right.
	\end{equation}
	with $\bv(0)=\buo$. The corresponding penalized system
	\begin{equation}
	\label{eq:NSE auxiliary 2 penalized}
	\left\{
	\begin{split}
	\frac{d\bve}{dt} + \tbB(\bve + \bze,\bve + \bze) - \nu\Delta\bve +  \nabla \rhoe &= 0, \mbox{ in } \RR^2,
	\\
	\dv\bve + \varepsilon\rhoe &= 0, \mbox{ in } \RR^2,
	\end{split}
	\right.
	\end{equation}
	with $\bve(0)=\buo$.
	
	The system \eqref{eq:NSE auxiliary 2} (resp. \eqref{eq:NSE auxiliary 2 penalized}) are interpreted as deterministic equations which solves $\bv$ (resp. $\bve$) for a given random process $\bz$ (resp.\ $\bze$).
	\section{Main algorithm and auxiliary results}
	\label{sec:auxiliary results}
	
	We consider a time discretization of \eqref{eq:NSE1} based on  the penalized system \cref{eq:penalized NSE1}. For that purpose we fix $M\in\NN$ and introduce an equidistant partition $I_k:=\{ t_\ell:1\leq \ell\leq M\}$ covering $[0,T]$ with mesh-size $k=T/M>0$, $t_0= 0$, and $t_M= T$. Here the increment $\Delta_\ell\bW := \bW(t_\ell) - \bW(t_{\ell-1})\sim\mathcal{N}(0,kQ)$ and we choose an uniform mesh size $k:= t_{\ell+1}-t_{\ell}$. For every $t\in[t_{\ell-1},t_{\ell}]$ and all $\bfi\in\WW^{1,2}_\per$, there hold $\PP\mbox{-a.s.}$,
	\begin{align}
	\label{eq:SNS vf dans H10}	
	\begin{split}
	(\bu(t_{\ell})-\bu(t_{\ell-1}),\bfi) + \nu\int_{t_{\ell-1}}^{t_\ell}(\nabla\bu(s),\nabla\bfi)ds
	\\
	+ \int_{t_{\ell-1}}^{t_\ell} \tb(\bu(s), \bu(s),\bfi)ds+\int_{t_{\ell-1}}^{t_\ell} (\nabla \pre(s),\bfi)ds &= \int_{t_{\ell-1}}^{t_\ell}\left(\bfi,d\bW(s)\right),
	\end{split}	
	\\
	\label{eq:incomp cond with chi}
	(\dv\bu(t_\ell),\chi) &=0.
	\end{align}
	Note that instead of $b$ we use $\tb$. We can switch between both without any confusion since for each $s \in [0,T]$, $\bu(s)\in\HH$.
	
	Now we discretize the penalized system \cref{eq:penalized NSE1} instead of the original equation and project the result into $\HH$. We derive the following algorithm:
	\begin{algo}[Main algorithm]\label{algo:penalty method}
		Assume $\bu^{\varepsilon,0}:=\buo$ with $\lVert \buo\rVert\leq C$. Find for every $\ell\in\{1,\ldots,M\}$ a pair of random variables $(\buel,\preel)$ with values in $\WW^{1,2}_\per\times L^2_\per$, such that we have $\PP$-a.s. 
		
		$\bullet$ Penalization:
		\begin{align}
		\label{eq:numerical scheme penalty vf1}
		\begin{split}
		(\tbuel - \buell,\bfi)+\nu k(\nabla\tbuel,\nabla\bfi)+k\tb(\tbuel, \tbuel,\bfi)
		\\
		+k(\nabla \tpreel,\bfi )+k(\nabla \fiell,\bfi )&= (\Delta_\ell\bW,\bfi),\;\forall\; \bfi\in \WW^{1,2}_\per,
		\end{split}
		\\
		\label{eq:numerical scheme penalty vf2}
		(\dv\tbuel,\chi) + \varepsilon (\tpreel,\chi)&=0,\;\forall\; \chi \in L^2_\per;
		\end{align}
		
		$\bullet$ Projection:
		\begin{align}
		\label{eq:numerical scheme projection 1}
		(\buel - \tbuel,\bfi) + \alpha k(\nabla(\fiel-\fiell),\bfi)  &=0,\;\forall\; \bfi\in \WW^{1,2}_\per,
		\\
		\label{eq:numerical scheme projection 2}
		(\dv \buel,\chi) &= 0,\;\forall\; \chi \in L^2_\per,
		\\\notag
		\preel&=\tpreel + \fiel + \alpha(\fiel -\fiell).
		\end{align}
	\end{algo}

	\begin{prop}\label{prop:existence and measurability}
		There exist iterates $\{\bu^\ell:1\leq\ell\leq M \}$ which solve \eqref{eq:numerical scheme penalty vf1} and \eqref{eq:numerical scheme penalty vf2} at each time-step. Moreover, for every integer $l$, with $1\leq\ell\leq M$, $\bu^\ell$ is a  $\cF_{t_\ell}$-measurable.
	\end{prop}
	\begin{proof}
		Let us fix $\omega\in\Omega$. We use Lax-Milgram fixed-point theorem to show the existence of a $\VV$-valued sequence $\{\buel:1\leq\ell\leq M  \}$.
		
		$\bullet$ Penalization:
		Since $\bu^{\varepsilon,0}$ and $\phi^0$ are given, and $\lvert\Delta_\ell \bW(\omega)\rvert_\mK<\infty$  for all  $\ell\in\{1,\ldots,M\}$, we  assume that $\tbu^{\varepsilon,1}(\omega),\ldots,\tbu^{\varepsilon,\ell-1}(\omega)$ are also given. To find  the pair of random variables $(\buel,\preel)$ in
		\cref{algo:penalty method} we need first to solve  a nonlinear, nonsymmetric variational problem. Therefore, let us  denote by $\sA$ the nonlinear operator from $\VV$ to $\VV'$ ($\VV'$: dual of $\VV$) defined by:
		\begin{equation}\label{eq:algo main variational form}
		\begin{split}
		\langle\sA\tbuel(\omega),\bw(\omega)\rangle=\tbuel(\omega)&+\nu (\nabla\tbuel(\omega),\nabla\bw^\ell(\omega)) \\&+ \tb(\tbuel(\omega),\tbuel(\omega),\bw^\ell(\omega)),\qquad\forall\;\bw^\ell((\omega))\in\VV.
		\end{split}
		\end{equation}
		Because $\tb$ satisfies the orthogonal property \eqref{eq:orthogonal property 1}, then putting $\bw^\ell=\tbuel$ in \eqref{eq:algo main variational form} we thus have
		\[
		\langle\sA\tbuel(\omega),\bw(\omega)\rangle\geq\nu\lVert \buel(\omega)\rVert_\VV^2.
		\]
		The operator $\sA$ is therefore $\VV$-elliptic and the Lax--Milgram theorem allows us to infer the existence of a unique solution of \eqref{eq:algo main variational form}. 
		
		$\bullet$ Projection: If we take $\bfi= \nabla \fiel$ in \eqref{eq:numerical scheme projection 1} we see that this step is actually a Poisson problem. Since $\buel$ is given from the previous step, the existence of a unique solution $\fiel$ is deduced from the ellipticity of the Laplacian operator.
		
		Since $\buel =\tbuel- \alpha k \nabla(\fiel - \phi^ {\varepsilon,\ell -1})$ and
		$\preel=\tpreel + \fiel + \alpha(\fiel - \fiell)$, the existence of a unique $\buel$ and  $\preel$ follows.
		
		The proof of the $\cF_{t_\ell}$-measurability of $\tbuel$ and $\fiel$ can be done in a similar fashion as in \cite{debouard2004a}, see also \cite{banas2014a}. Since $\buel$ (resp. $\preel$) are obtained from $\tbuel$ (resp. $\fiel$) we also obtain their $\cF_{t_\ell}$-measurability.
	\end{proof}
	
	\subsection{Stability}\label{subsec:stability}
	This section is inspired by \cite[Lemma 3.1]{brzezniak2013finite} and \cite[Lemma 3.1]{carelli2012rates}. Here we consider a coupled system, the first one is derived from the penalization and the second one is a projection step.
	
	\medskip

	\begin{lem}
		\label{lem:stability}Let $\phi^{\varepsilon,0} = 0$.
		Suppose that \cref{S1,S4} are valid with $\lVert \bu^0\rVert\leq C$. Then, there exists a positive constant $C=C(L,T, \bu^0,\nu)$ so that for every $\varepsilon>0$ and $\alpha>1$, the iterates $\{\buel:1\leq\ell\leq M\}$ solving \cref{algo:penalty method}  and the intermediary iterates $\{ \tbuel:1\leq\ell\leq M\}$, $\{\tpreel:1\leq\ell\leq M\}$, and $\{ \fiel:1\leq\ell\leq M\}$ satisfy for $q= 1$ or $q=2$:
		\begin{itemize}
			\item[$(i)$] $\displaystyle\nu \EE\left(k\sum_{\ell = 1}^{M}\lVert\nabla\tbuel\rVert^2\lVert\buel\rVert^{2^q-2}\right)\leq C $,
			\item[$(ii)$]  $\displaystyle k^2 \EE\max_{1\leq m\leq M}\lVert \nabla\fiem\rVert^2\lVert \buem\rVert^{2^q-2}
			+ {\varepsilon}\EE\left(k\sum_{\ell = 1}^{M}\lVert\tpreel\rVert \lVert\buel\rVert^{2^q-2}\right)\leq C$,
			\item[$(iii)$] $\displaystyle\EE\max_{1\leq m\leq M}\lVert\buem\rVert^{2^q}+\nu \EE\left(k\sum_{\ell = 1}^{M}\lVert\nabla\buel\rVert^2\lVert\buel\rVert^{2^q-2}\right) \leq C$.
		\end{itemize}
		
	\end{lem}
	\begin{proof}
		The proof consits of three steps. First, we give some preparatory estimates.
		Then, we handle the case $q=1$, and, fially, we handle the case $q=2$.
		\begin{steps}
			\item 	{\bf Preparatory estimate.}				
			We take $\bfi= 2\tbuel$ in \cref{eq:numerical scheme penalty vf1} and $\chi=\dv\tbuel$ in \cref{eq:numerical scheme penalty vf2}, and use the orthogonal property~\eqref{eq:orthogonal property 1} of $\tb$, to get
			\begin{align}
			\label{eq:numerical scheme penalty vf1 1}
			\begin{split}
			(\tbuel - \buell,\tbuel)+ 2\nu k\lVert\nabla\tbuel\rVert^2
			+2k(\nabla \tpreel,\buel)
			\\
			+2k(\nabla \phi^{\varepsilon,\ell-1},\tbuel)&= 2(\Delta_\ell\bW,\tbuel),
			\end{split}	
			\\\notag
			(\nabla\tpreel,\tbuel)&=\frac{1}{\varepsilon}\lVert \dv\tbuel\rVert^2.
			\end{align}
			Using the algebraic identity
			\begin{equation}
			\label{eq:algebraic identity}
			2(a-b)a = a^2 -b^2 + (a-b)^2
			\end{equation}
			in \eqref{eq:numerical scheme penalty vf1 1} we obtain
			\begin{equation}
			\label{eq:penalty stability}
			\begin{split}
			\lVert\tbuel\rVert^2-\lVert\buell\rVert^2 + \lVert\tbuel-\tbuell \rVert^2
			&+ 2 \nu k\lVert\nabla\tbuel\rVert^2 + 2{\varepsilon}k\lVert \tpreel\rVert^2
			\\
			&+ 2k(\nabla\fiell,\tbuel)
			=  2(\Delta_\ell\bW,\tbuel).
			\end{split}
			\end{equation}
			Let $\alpha>0$. We take $\bfi = 2\tbuel$ in \eqref{eq:numerical scheme projection 1} and obtain
			\begin{equation}
			\label{eq:penalty stability5}
			\frac{\alpha-1}{\alpha} \left(\lVert \buel\rVert^2-\lVert \tbuel\rVert^2 + \lVert \buel -\tbuel\rVert^2\right) = 0.
			\end{equation}
			Then, we take $\bfi = \buel + \tbuel$ in \eqref{eq:numerical scheme projection 1} and obtain
			\begin{equation}
			\label{eq:penalty stability6}
			\frac{1}{\alpha}\left(\lVert \buel\rVert^2-\lVert \tbuel\rVert^2\right) + \frac{k}{2}(\nabla(\fiel -\fiell),\tbuel) =0.
			\end{equation}
			Collecting together  \eqref{eq:penalty stability} to \eqref{eq:penalty stability6}  we obtain
			\begin{equation}
			\label{eq:penalty stability7}
			\begin{split}
			\lVert\buel\rVert^2-\lVert\buell\rVert^2 &+ \lVert\tbuel-\tbuell \rVert^2+\frac{\alpha-1}{\alpha} \lVert \buel -\tbuel\rVert^2+ 2 \nu k\lVert\nabla\tbuel\rVert^2
			\\
			&+ 2{\varepsilon}k\lVert \tpreel\rVert^2 + k(\nabla(\fiell+\fiel),\tbuel)
			\leq  2(\Delta_\ell\bW,\tbuel).
			\end{split}
			\end{equation}
			We take $\bfi = \nabla(\fiel + \fiell)$ in \eqref{eq:numerical scheme projection 1} and obtain
			\begin{equation*}
			(\nabla(\fiel +\fiell),\tbuel) = \alpha k\lVert \nabla\fiel\rVert^2-\alpha k\lVert \nabla\fiell\rVert^2.
			\end{equation*}
			This implies,
			\begin{equation}
			\label{eq:penalty stability9}
			\begin{split}
			\lVert\buel\rVert^2&-\lVert\buell\rVert^2 + \lVert\tbuel-\tbuell \rVert^2+\frac{\alpha-1}{\alpha} \lVert \buel -\tbuel\rVert^2+ 2 \nu k\lVert\nabla\tbuel\rVert^2
			\\
			&+ 2{\varepsilon}k\lVert \tpreel\rVert^2 + \alpha k^2\lVert \nabla\fiel\rVert^2-\alpha k^2\lVert \nabla\fiell\rVert^2
			\leq  2(\Delta_\ell\bW,\tbuel).
			\end{split}
			\end{equation}
			
			\item	{\textbf{Case $q=1$.}} Summing \eqref{eq:penalty stability9} from $\ell = 1$ to $\ell= m$, we get
			\begin{equation}
			\label{eq:penalty stability 2}
			\begin{split}
			\lVert \buem\rVert^2 &+\sum_{\ell = 1}^{m} \lVert\tbuel-\tbuell \rVert^2+(\tfrac{\alpha-1}{\alpha} )\sum_{\ell = 1}^{m}\lVert \buel -\tbuel\rVert^2+ 2 \nu k\sum_{\ell = 1}^{M}\lVert\nabla\tbuel\rVert^2
			\\
			& + 2{\varepsilon}k\sum_{\ell = 1}^{m}\lVert \tpreel\rVert^2+\alpha k^2\lVert \nabla\fiem\rVert^2
			\leq \lVert \bu^0\rVert^2 +  2\sum_{\ell = 1}^{m}(\Delta_\ell\bW,\tbuel).
			\end{split}
			\end{equation}
			The last term of the right side can be splitted as follows,
			\begin{align}
			\notag
			\noise^{\varepsilon,m}_1:&=2\sum_{\ell= 1}^{m}(\Delta_\ell\bW,\tbuel)=2\sum_{\ell= 1}^{m}(\Delta_\ell\bW,\tbuel-\tbuell)+2\sum_{\ell= 1}^{m}(\Delta_\ell\bW,\tbuell).
			\intertext{Let $ \dun>0$ be an arbitrary number. Applying the  Young inequality to the  first term on the right side, we get}
			\label{eq:penalty stability 4}
			\ldots&\leq C(\dun) \sum_{\ell= 1}^{m}\lVert \Delta_\ell\bW\rVert^2+\dun\sum_{\ell= 1}^{m}\lVert\tbuel-\tbuell\rVert^2+2\sum_{\ell= 1}^{m}(\Delta_\ell\bW,\tbuell).
			\end{align}
			Taking first the maximum of \eqref{eq:penalty stability 4}  over $1\leq m\leq M$, and then  the expectation, exactly with this order, give the following estimate
			\begin{equation}
			\label{eq:penalty stability 5}
			\begin{split}
			\EE\max_{1\leq m\leq M}\noise^{\varepsilon,m}_1\leq C(\dun) \sum_{\ell= 1}^{M}\EE\lVert \Delta_\ell\bW\rVert^2&+\dun\sum_{\ell= 1}^{M}\EE\lVert\tbuel-\tbuell\rVert^2
			\\
			&+2\EE\max_{1\leq m\leq M}\sum_{\ell= 1}^{m}(\Delta_\ell\bW,\tbuell).
			\end{split}
			\end{equation}
			It follows from \eqref{eq:increment of a wiener process}, that $\EE\lVert \Delta_\ell\bW\rVert^2=k$. By applying successively the Burkholder--David--Gundy inequality, the H\"older inequality, and the Young inequality to the last term of \eqref{eq:penalty stability 5}, we obtain
			\begin{align*}
			\EE\max_{1\leq m\leq M}\noise^{\varepsilon,m}_1&\leq C(\dun, T)+\dun\sum_{\ell= 1}^{M}\EE\lVert\buel-\buell\rVert^2+\EE\left(\sum_{\ell= 1}^{M}k\lVert\buell\rVert^2\right)^{1/2}
			\\
			&\leq C({\dun,T ,\bu^0})+\dun\sum_{\ell= 1}^{M}\EE\lVert\buel-\buell\rVert^2+\dun \EE\max_{1\leq \ell\leq M }\lVert\buel\rVert^2.
			\end{align*}
			Now, taking the maximum of \eqref{eq:penalty stability 2} over $1\leq m\leq M$, and, then,  expectation,  give the following estimate,
			\begin{equation}
			\label{eq:penalty stability 7}
			\begin{split}
			&\EE\max_{1\leq m\leq M}\left\{\lVert \buem\rVert^2 + \alpha k^2 \lVert \nabla\fiem\rVert^2\right\}
			\\
			&+\EE\sum_{\ell = 1}^{m} \lVert\tbuel-\tbuell \rVert^2+(\tfrac{\alpha-1}{\alpha})\EE \sum_{\ell = 1}^{m}\lVert \buel -\tbuel\rVert^2
			\\
			&+ \nu \EE \left(k\sum_{\ell = 1}^{M}\lVert\nabla\tbuel\rVert^2\right) + {\varepsilon}\EE \left(k\sum_{\ell = 1}^{m}\lVert \tpreel\rVert^2\right)
			\leq \lVert \bu^0\rVert^2 +\EE\max_{1\leq m\leq M}\noise^{\varepsilon,m}_1.
			\end{split}
			\end{equation}
			The terms with $\lVert\tbuel-\tbuell\rVert^2$ and $\max_{1\leq \ell\leq M }\lVert\tbuel\rVert^2$ of \eqref{eq:penalty stability 5} are absorbed by the left hand side of \eqref{eq:penalty stability 7} which leads to
			\begin{equation}
			\label{eq:penalty stability 8}
			\begin{split}
			&(1-\dun)\EE\max_{1\leq m\leq M}\lVert \buem\rVert^2 + \alpha k^2 \EE\max_{1\leq m\leq M}\lVert \nabla\fiem\rVert^2
			\\[5pt]
			&+(1-\dun)\EE\sum_{\ell = 1}^{m} \lVert\tbuel-\tbuell \rVert^2+(\tfrac{\alpha-1}{\alpha})\EE \sum_{\ell = 1}^{m}\lVert \buel -\tbuel\rVert^2
			\\[3pt]
			&+ \nu \EE \left(k\sum_{\ell = 1}^{M}\lVert\nabla\tbuel\rVert^2\right) +{\varepsilon}\EE \left(k\sum_{\ell = 1}^{m}\lVert \tpreel\rVert^2\right)
			\leq C(\dun, T, \bu^0).
			\end{split}
			\end{equation}
			The parameters $\alpha$ and $\dun$  are chosen such that the left hand side stays positive. Thus, we chose  $\alpha>1$ and $0<\dun<1$.
			
			\item	{\textbf{Case $q=2$.}}
			We multiply \eqref{eq:penalty stability9} by $2\lVert \buel\rVert^2$ and use again the algebraic identity \eqref{eq:algebraic identity} to give
			\begin{equation}
			\label{eq:penalty stability2}
			\begin{split}
			&\lVert\buel\rVert^4-\lVert\buell\rVert^4
			+2\lVert\tbuel-\tbuell\rVert^2\lVert\buel\rVert^2
			\\[5pt]
			&+\tfrac{\alpha-1}{\alpha} \lVert \buel -\tbuel\rVert^2\lVert \buel\rVert^2+ 4\nu k\lVert\nabla\tbuel\rVert^2\lVert\buel\rVert^2
			+ 4{\varepsilon}k\lVert\tpreel\rVert^2 \lVert\buel\rVert^2
			\\[5pt]
			&+2\alpha k^2\lVert \nabla\fiel\rVert^2\lVert \buel\rVert^2-2\alpha k^2\lVert \nabla\fiell\rVert^2\lVert \buel\rVert^2
			=  2(\Delta_\ell\bW,\tbuel)\lVert\buel\rVert^2.
			\end{split}
			\end{equation}
			On the left hand side we use the same calculation that we used on the term  $\noise^{\varepsilon,m}_1$. In particular,
			we compute
			\begin{align}
			\notag
			\noise^{\varepsilon,m}_2 &:= 2\sum_{\ell= 1}^{m}(\Delta_\ell\bW,\tbuel)\lVert \buel\rVert^2
			\\
			\label{eq:penalty stability 6}
			&\leq C(\dun) \sum_{\ell= 1}^{m}\lVert \Delta_\ell\bW\rVert^2+\dun\sum_{\ell= 1}^{m}\lVert\tbuel-\tbuell\rVert^2\lVert \buel\rVert^2
			\\\notag
			&\qquad\qquad\qquad\qquad\quad+2\sum_{\ell= 1}^{m}(\Delta_\ell\bW,\tbuell)\lVert \buel\rVert^2.
			\end{align}
			In the next step, we first take  the maximum of \eqref{eq:penalty stability 6}  over $1\leq m\leq M$, and, then, we take the  expectation, exactly with this order. Now, applying  the Young inequality, and the H\"older inequality, and using \eqref{eq:penalty stability 8} to bound some terms, we can find a constant $C=C(\dun,\ddeu,L,T, \bu^0,\nu)>0$ such that
			\begin{equation}
			\label{eq:penalty stability 10}
			\begin{split}
			\EE\max_{1\leq m\leq M}\noise^{\varepsilon,m}_2
			\leq C&+\dun\sum_{\ell= 1}^{M}\lVert\tbuel-\tbuell\rVert^2\lVert \buel\rVert^2+\ddeu\EE\max_{1\leq \ell \leq M}\lVert \buel\rVert^4.
			\end{split}
			\end{equation}
			Summing up \eqref{eq:penalty stability 10} for $\ell=1$ to $\ell=m$, taking the maximum over $1\leq m\leq M$, and taking the expectation in \eqref{eq:penalty stability2}	 we have
			\begin{equation}
			\label{eq:penalty stability3}
			\begin{split}
			&\EE\max_{1\leq m\leq M}\left\{\lVert\buem\rVert^4+\alpha k^2 \lVert \nabla\fiel\rVert^2\lVert \buel\rVert^2\right\}
			+\EE\sum_{\ell = 1}^{M}\lVert\tbuel-\tbuell\rVert^2\lVert\buel\rVert^2
			\\
			&+(\tfrac{\alpha-1}{\alpha}) \EE\sum_{\ell = 1}^{M}\lVert \buel -\tbuel\rVert^2\lVert \buel\rVert^2
			+ \nu \EE\left(k\sum_{\ell = 1}^{M}\lVert\nabla\tbuel\rVert^2\lVert\buel\rVert^2\right)
			\\
			&\hspace{0pt}+ {\varepsilon}\EE\left(k\sum_{\ell = 1}^{M}\lVert\tpreel\rVert^2 \lVert\buel\rVert^2\right)
			\\
			&\leq C(\dun,\ddeu,L,T, \bu^0,\nu)+\dun\sum_{\ell= 1}^{M}\lVert\tbuel-\tbuell\rVert^2\lVert \buel\rVert^2+\ddeu\EE\max_{1\leq \ell \leq M}\lVert \buel\rVert^4.
			\end{split}
			\end{equation}
			The terms with $\lVert \buel\rVert^4$ and $\lVert\tbuel-\tbuell\rVert^2\lVert \buel\rVert^2$ are absorbed by the left side of \eqref{eq:penalty stability3}. Therefore, we get
			\begin{equation*}
			\begin{split}
			&\EE\max_{1\leq m\leq M}\left\{(1-\ddeu)\lVert\buem\rVert^4+\alpha k^2 \lVert \nabla\fiem\rVert^2\lVert \buem\rVert^2\right\}
			\\
			&+(1-\dun)\EE\sum_{\ell = 1}^{M}\lVert\tbuel-\tbuell\rVert^2\lVert\buel\rVert^2
			+(\tfrac{\alpha-1}{\alpha}) \EE\sum_{\ell = 1}^{M}\lVert \buel -\tbuel\rVert^2\lVert \buel\rVert^2\\
			&+ \nu \EE\left(k\sum_{\ell = 1}^{M}\lVert\nabla\tbuel\rVert^2\lVert\buel\rVert^2\right)
			+{\varepsilon}\EE\left(k\sum_{\ell = 1}^{M}\lVert\tpreel\rVert^2 \lVert\buel\rVert^2\right)
			\leq C(\dun,\ddeu,L,T, \bu^0,\nu).
			\end{split}
			\end{equation*}
			We conclude by choosing $\alpha,\dun,$ and $\ddeu$ such that, $(\alpha-1) >0$, $(1-\dun)>0$, and $(1-\ddeu)>0$.
		\end{steps}

	\end{proof}
	
	In the next lemma we use the LBB inequality (see \cite{babuska1973the,brezzi1974on})
	\begin{equation}
	\label{eq:LBB condition}
	\lVert \pre\rVert\leq C\sup_{\bfi\in\WW^{1,2}}\frac{(\nabla\pre,\bfi)}{\lVert \bfi\rVert_1}
	\end{equation}
	to transfer the estimate from the velocity fields $\buel$ to the pressure fields $\preel$.
	
	We start with a direct discretization of \eqref{eq:NSE1} which leads to the following algorithm:
	
	\medskip
	
	\begin{algo}\label{algo:direct discretization}
		Assume $\bu^0:=\buo$ with $\lVert \bu^0\rVert\leq C$. Find for every $\ell\in\{1,\ldots,M\}$ a pair of random variables $(\bu^\ell,\pre^\ell)$ with values in $\VV\times L^2_\per$, such that $\PP$-a.s.
		\begin{align}
		\label{eq:direct discretization1}
		\begin{split}
		(\bu^{\ell} - \bu^{\ell-1},\bfi)+\nu k(\nabla\bu^{\ell},\nabla\bfi)+k\tb(\bu^\ell, \bu^\ell,\bfi)\\+k(\nabla \pre^{\ell},\bfi )&= (\Delta_\ell\bW,\bfi),\quad \forall\; \bfi\in \WW^{1,2}_\per,
		\end{split}\\
		\label{eq:direct discretization2}
		(\dv\bu^{\ell},\chi)&=0,\quad \forall\; \chi \in L^2_\per.
		\end{align}
	\end{algo}	
	
	We define the sequences of errors $\sel =\bul-\buel, \tsel =\bul-\tbuel $, and $\sql = \prel-\preel$. We subtract \eqref{eq:numerical scheme penalty vf1} and \eqref{eq:numerical scheme penalty vf2} from \eqref{eq:direct discretization1} and \eqref{eq:direct discretization2},
	and get
	\begin{align}
	\label{eq:perturbation analysis 1}
	\begin{split}
	(\sel - \sell,\bfi)&+\nu k(\nabla\tsel,\nabla\bfi)
	\\
	&+k(\nabla \sql,\bfi )=k\tb(\tbuell,\tbuel,\bfi)-k\tb(\bul,\bul,\bfi),\; \forall\; \bfi\in \WW^{1,2}_\per.
	\end{split}
	\end{align}
	\begin{lem}\label{lem:stability for pressure}
		Under the assumption of Lemma~\ref{lem:stability}, there exists a  constant $C=C(L,T,\buo)>0$ such that for every $\varepsilon>0$, the iterates $\{ \preel:1\leq\ell\leq M\}$ solving \cref{algo:penalty method} satisfies
		\begin{equation*}
		\EE\left( k\sum_{\ell = 1}^{M}\lVert \preel\rVert^2\right)\leq C.
		\end{equation*}
		
	\end{lem}
	\begin{proof}
		Since $(\sel -\sell)\in \cD(\bA^{-1})$, we can take $\bfi = \bA^{-1}(\sel -\sell)$ in
		\eqref{eq:perturbation analysis 1} and use Proposition~\ref{prop:equivalence of norms}.
		Then we apply the Young inequality, and use estimate \eqref{eq:estimate 3 of tb} of $\tb$. This  leads to the following results:
		\\[1pt]
		\begin{itemize}
			\item [$i)$]$c_2\lVert \sel - \sell \rVert_{-1}^2\leq (\sel - \sell,\bA^{-1}(\sel -\sell)),$
			\\[1pt]
			\item [$ii)$]$\nu k(\nabla\sel,\nabla\bA^{-1}(\sel -\sell))\leq  C(\dun)\nu^2 k^2\lVert\sel\rVert_1^2 + \dun\lVert \sel - \sell \rVert_{-1}^2,$
			\\[1pt]
			\item [$iii)$]$k(\nabla \sql,\bA^{-1}(\sel -\sell) ) = 0,$
			\\[1pt]
			\item [$iv)$]$k\tb(\tbuell,\tbuel,\bA^{-1}(\sel -\sell))\leq C(L,\dun)\frac{k^2}{2} \lVert \tbuell\rVert^2 \lVert \tbuell\rVert_1^2+C(L,\dun)\frac{k^2}{2} \lVert \tbuel\rVert^2 \lVert \tbuel\rVert_1^2+ \dun\lVert \sel -\sell\rVert^2_{-1},$
			\\[1pt]
			\item [$v)$]$k\tb(\bul,\bul,\bA^{-1}(\sel -\sell))\leq C(L,\dun)k^2 \lVert \bul\rVert^2 \lVert \bul\rVert_1^2 +\dun \lVert \sel -\sell\rVert^2_{-1}.$
			\\[1pt]
		\end{itemize}
		Fixing  $ \dun>0$ such that $(c_2-3\dun)>0$, and collecting  $i)$, $ii)$, $iii)$, $iv)$, and $v)$, we obtain
		\begin{align*}
		&(c_2-3\dun)\EE\sum_{\ell =1}^{M}\lVert \sel - \sell \rVert_{-1}^2\leq C(L,\dun,\nu) k\EE\left(k\sum_{\ell = 1}^{M}\lVert\buel\rVert_1^2\right) + k\EE\left(k\sum_{\ell = 1}^{M}\lVert\bul\rVert_1^2\right)
		\\
		& +k\EE\left(k \sum_{\ell =1}^{M}\lVert \tbuell\rVert^2 \lVert \tbuell\rVert_1^2\right)+ k\EE\left(k \sum_{\ell =1}^{M}\lVert \tbuel\rVert^2 \lVert \tbuel\rVert_1^2\right) + k\EE\left(k\sum_{\ell =1}^{M} \lVert \bul\rVert^2 \lVert \bul\rVert_1^2\right).
		\end{align*}
		By Lemma~\ref{lem:stability} and \cite[Lemma 3.1 (iii)]{brzezniak2013finite} we obtain
		\begin{equation}
		\label{eq:stability pressure4}
		\EE\sum_{\ell =1}^{M}\lVert \sel - \sell \rVert_{-1}^2\leq C(T,L,\bu^0)k.
		\end{equation}
		Now we rearrange \eqref{eq:perturbation analysis 1} and get
		\begin{align}\label{eq:stability pressure1}
		k(\nabla \sql,\bfi ) = -(\sel - \sell,\bfi) - \nu k(\nabla\tsel,\nabla\bfi) + k\tb(\tbuell,\tbuel,\bfi)-k\tb(\bul,\bul,\bfi).
		\end{align}
		With the skew symmetry property of $\tb$ (see \eqref{eq:skew-symmetry}) and the estimate \eqref{eq:estimate 3 of tb},
		identity  \eqref{eq:stability pressure1} becomes
		\begin{align*}
		\frac{k(\nabla \sql,\bfi )}{\lVert \bfi\rVert_1}\leq \lVert \sel - \sell\rVert_{-1} &+ \nu k\lVert \nabla\tsel\rVert + C(L)k\lVert \tbuell\rVert \lVert \tbuell\rVert_1
		\\
		&+ C(L)k\lVert \buel\rVert \lVert \buel\rVert_1 +C(L)k\lVert \bul\rVert \lVert \bul\rVert_1.
		\end{align*}
		Using the inequality \eqref{eq:LBB condition}, we have
		\begin{align*}
		k^2\lVert \sql\rVert^2 \leq C\lVert \sel - \sell\rVert_{-1}^2 &+ \nu^2 k^2\lVert \nabla\tsel\rVert^2 +C(L)k^2\lVert \tbuell\rVert^2 \lVert \tbuell\rVert_1^2
		\\
		&+ C(L)k^2\lVert \tbuel\rVert^2 \lVert \tbuel\rVert_1^2 +C(L)k^2\lVert \bul\rVert^2 \lVert \bul\rVert_1^2.
		\end{align*}
		Summing up for $\ell =1$ to $\ell = M$, and taking expectation, we obtain
		\begin{align*}
		k\EE\left(k\sum_{\ell =1}^{M}\lVert \sql\rVert^2\right) \leq &C\EE\sum_{\ell =1}^{M}\lVert \sel - \sell\rVert_{-1}^2 + \nu^2 k\EE\left(k\sum_{\ell =1}^{M}\lVert \nabla\tbuel\rVert^2\right)
		\\
		&+ \nu^2 k\EE\left(k\sum_{\ell =1}^{M}\lVert \nabla\bul\rVert^2\right) + C(L)k\EE\left(k\sum_{\ell =1}^{M}\lVert \tbuell\rVert^2 \lVert \tbuell\rVert_1^2\right) \\
		&+ C(L)k\EE\left(k\sum_{\ell =1}^{M}\lVert \tbuel\rVert^2 \lVert \tbuel\rVert_1^2\right) +C(L)k\EE\left(k\sum_{\ell =1}^{M}\lVert \bul\rVert^2 \lVert \bul\rVert_1^2\right).
		\end{align*}
		From Lemma~\ref{lem:stability}, \cite[Lemma 3.1 (iii)]{brzezniak2013finite}, and estimate \eqref{eq:stability pressure4}, we obtain
		\begin{equation*}
		\EE\left(k\sum_{\ell =1}^{M}\lVert \sql\rVert^2\right) \leq C(T,L,\nu,\bu^0).
		\end{equation*}
		The Minkowsky inequality and Poincar\'e inequality imply
		\begin{align*}
		\EE\left(k\sum_{\ell =1}^{M}\lVert \preel\rVert^2\right) 	
		&\leq C(T,L,\nu,\bu^0) + C(L)\EE\left(k\sum_{\ell =1}^{M}\lVert \nabla\prel\rVert^2\right).
		\end{align*}
		We finish the proof  with  using \cite[Lemma 3.2 (i)]{carelli2012rates}, where the authors proved that
		\[
		\EE\left(k\sum_{\ell =1}^{M}\lVert \nabla\prel\rVert^2\right)\leq C(T).
		\]
		
	\end{proof}
	\subsection{Auxiliary error estimates}
	We start with \cref{algo:auxilary pbm for z}. Let $\bz=\left\lbrace \bz(t,\cdot): t\in[t_{\ell-1},t_\ell]\right\rbrace$ be the strong solution of \eqref{eq:NSE auxiliary} as defined in Definition~\ref{def:strong solution} and $\pi=\left\lbrace \pi(t,\cdot): t\in[t_{\ell-1},t_\ell]\right\rbrace$ the associated pressure, i.e.\ for every $t\in[t_{\ell-1},t_\ell]$,  all $\bfi\in\WW^{1,2}$, and all $\chi\in L^2_\per$, we have  $\PP\mbox{-a.s.}$
	\begin{align}
	\label{eq:auxiliary 1 strong 1}
	\begin{split}
	(\bz(t_{\ell})-\bz(t_{\ell-1}),\bfi) + \nu\int_{t_{\ell-1}}^{t_\ell} (\nabla\bz(s),\nabla\bfi)ds
	\\
	+ \int_{t_{\ell-1}}^{t_\ell} (\nabla\pi(s),\bfi)ds &= \int_{t_{\ell-1}}^{t_\ell}\left(\bfi,d\bW(s)\right),
	\end{split}
	\\
	\label{eq:auxiliary 1 strong 2}
	(\dv\bz(t_{\ell}),\chi) &= 0.
	\end{align}
	\noindent
	For \eqref{eq:auxiliary 1 strong 1} and \eqref{eq:auxiliary 1 strong 2} we have the following algorithm:
	\begin{algo}[First auxiliary algorithm]\label{algo:auxilary pbm for z}
		Let $\bz^0:=0$. Find for every $\ell\in\{1,\ldots,M\}$ a pair of random variables $(\bz^\ell,\pi^\ell)$ with values in $\WW^{1,2}_\per\times L^2_\per$, such that we have $\PP$-a.s.
		
		$\bullet$ Penalization:
		\begin{align}
		\label{eq:auxilary pbm for z 1}
		\begin{split}
		(\tbzel - \bzell,\bfi)+\nu k(\nabla\tbzel,\nabla\bfi)+k(\nabla \tpiel,\bfi )
		\\
		+k(\nabla \xiell,\bfi )&= (\Delta_\ell\bW,\bfi),\quad \forall\;\bfi\in \WW^{1,2}_\per,
		\end{split}
		\\
		\label{eq:auxilary pbm for z 2}
		(\dv\tbzel,\chi) + \varepsilon (\tpiel,\chi)&=0,\quad \forall\; \chi \in L^2_\per.
		\end{align}
		
		$\bullet$ Projection:
		\begin{align}
		\label{eq:zprojection 1}
		(\bzel - \tbzel,\bfi) + \alpha k(\nabla(\xiel-\xiell),\bfi)  &=0,\;\forall\; \bfi\in \WW^{1,2}_\per,
		\\
		\label{eq:zprojection 2}
		(\dv \bzel,\chi) &= 0,\;\forall\; \chi \in L^2_\per,
		\\\notag
		\piel &=\tpiel + \xiel + \alpha(\xiel -\xiell).
		\end{align}
	\end{algo}
	Define the errors $\tel= \bz(t_\ell)-\tbzel$, $\el= \bz(t_\ell)-\bzel$ and $\vpil = \pi(t_\ell)-\piel$. We subtract \eqref{eq:auxilary pbm for z 1} from \eqref{eq:auxiliary 1 strong 1} to get
	\begin{align}\label{eq:auxilary pbm for z error estimate 1}
	\begin{split}
	(\tel-\eell,\bfi) &+  \nu\int_{t_{\ell-1}}^{t_\ell}(\nabla(\bz(s)-\tbzel),\nabla\bfi)ds
	\\
	&+\int_{t_{\ell-1}}^{t_\ell}(\nabla(\pi(s)-\tpiel-\xiell),\bfi)ds  = 0,
	\end{split}		
	\end{align}
	and choose $\chi=\dv \bfi$ in \eqref{eq:auxilary pbm for z 2} to get
	\begin{equation}
	\label{eq:pseudo-compressible constraint for z}
	-(\nabla\piel,\bfi) = \frac{1}{\varepsilon}(\nabla\dv \tbzel,\bfi).
	\end{equation}
	Thanks to the following identities
	\begin{align}
	\label{eq:identification dz 1}
	(\nabla(\bz(s)-\tbzel),\nabla\bfi)&=  (\nabla\te^\ell, \nabla\bfi)+ \left(\nabla (\bz(s)- \bz(t_\ell)),\nabla\bfi\right) \mbox{ and }
	\\
	\label{eq:identification dz 2}
	(\nabla(\pi(s)-\piel-\xiell),\bfi) &= (\nabla\pi(s),\bfi) + \frac{1}{\varepsilon}(\nabla\dv\tbzel,\bfi) - (\nabla\xiell,\bfi),
	\end{align}
	the equation \eqref{eq:auxilary pbm for z error estimate 1} is reduced to
	\begin{equation}
	\label{eq:auxilary pbm for z error estimate 3}
	\begin{split}
	(\te^\ell-\e^{\ell-1},\bfi) +  \nu k(\nabla\te^\ell, \nabla\bfi)
	&- \frac{k}{\varepsilon}(\nabla\dv\tbzel,\bfi)
	\\
	&- k(\nabla\xiell,\bfi) =R_\ell^\bz(\bfi)-\int_{t_{\ell-1}}^{t_\ell}(\nabla\pi(s),\bfi)ds,
	\end{split}
	\end{equation}
	where
	\[
	\begin{split}
	R_\ell^\bz(\bfi) =\nu \int_{t_{\ell-1}}^{t_\ell}\left(\nabla (\bz(t_\ell)- \bz(s)),\nabla\bfi\right)ds.
	\end{split}
	\]
	To \eqref{eq:auxilary pbm for z error estimate 3} we associate the following projection equation
	\begin{equation}
	\label{eq:projection error on z}
	\left\{
	\begin{split}
	(\el-\tel,\bfi) &=k\alpha(\nabla(\xiel - \xiell),\bfi),
	\\
	\dv\el &= 0.
	\end{split}
	\right.
	\end{equation}
	\begin{lem}\label{lem:auxiliary pbm for z error estimate}Let $\alpha>1$ and $0<\eta<1/2$. For every $\varepsilon>0$, there exists a constant $C = C(T,\nu,\eta)>0$ such that
		\begin{align}\label{eq:auxiliary pbm for z error estimate 0}
		\EE\max_{1<m\leq M}\lVert \e^m\rVert^2+\nu\EE\left(k\sum_{\ell=1}^{M}\lVert \nabla\el\rVert^2
		\right)&\leq C(k^\eta + \varepsilon).
		\end{align}
		
	\end{lem}
	\begin{proof}
		We take $\bfi= 2\te^\ell$ in \eqref{eq:auxilary pbm for z error estimate 3}. Then we use the algebraic identity \eqref{eq:algebraic identity} and the fact that $\dv\bz(t_\ell) = 0$ to get
		\begin{equation}\label{eq:auxiliary pbm for z error estimate first}
		\begin{split}
		\lVert \te^\ell\rVert^2-\lVert\eell \rVert^2&+\lVert \te^\ell-\e^{\ell-1} \rVert^2 +2\nu k\lVert \nabla\te^\ell\rVert^2+ \frac{2k}{\varepsilon}\lVert\dv\te^\ell\rVert^2\\
		&= 2k(\nabla\xiell,\te^\ell)+R_\ell^\bz(2\te^\ell) +\int_{t_{\ell-1}}^{t_\ell}(\dv\te^\ell,\pi(s))ds.
		\end{split}
		\end{equation}
		Let us take $\bfi = \tel +\el$ in \eqref{eq:projection error on z} to get
		\begin{align}
		\label{eq:projection error on z 1}
		\frac{\alpha-1}{\alpha}\left(\lVert \e^\ell\rVert^2 -\lVert\te^\ell \rVert^2 + \lVert\te^\ell-\e^\ell \rVert^2\right) &= 0,
		\\
		\label{eq:projection error on z 2}
		\frac{1}{\alpha}\left(\lVert \e^\ell\rVert^2  -\lVert \te^\ell\rVert^2\right) &= \frac{k}{2}(\nabla(\xiel -\xiell),\te^\ell).
		\end{align}
		Collecting \eqref{eq:auxiliary pbm for z error estimate first} to \eqref{eq:projection error on z 2} together, we arrive at
		\begin{align}\label{eq:auxiliary pbm for z error estimate first1}
		\begin{split}
		\lVert \e^\ell\rVert^2-\lVert\e^{\ell-1} \rVert^2
		&+\lVert \te^\ell-\e^{\ell-1} \rVert^2+(\tfrac{\alpha-1}{\alpha})\lVert\te^\ell-\e^\ell \rVert^2 +2\nu k\lVert \nabla\te^\ell\rVert^2
		\\
		+ \frac{2k}{\varepsilon}\lVert\dv\te^\ell\rVert^2&\leq 2{k}(\nabla(\xiel+\xiell),\te^\ell)+R_\ell^\bz(2\te^\ell) +\int_{t_{\ell-1}}^{t_\ell}(\dv\te^\ell,\pi(s))ds.
		\end{split}
		\end{align}
		First, notice that from \eqref{eq:projection error on z} it follows
		\[
		\te^\ell = \e^\ell -k\alpha\nabla(\xiel -\xiell).
		\]
		Therefore, we have
		\begin{align}\label{eq:auxiliary pbm for z error estimate first2}
		2{k}(\nabla(\xiel+\xiell),\te^\ell)
		=2\alpha{k}^2\lVert \xiell\rVert^2- 2\alpha{k}^2\lVert\xiel \rVert^2.
		\end{align}
		Secondly, applying the Young inequality to $R_\ell^z(2\te^\ell)$, we get
		\begin{align}\label{eq:auxiliary pbm for z error estimate 1 for R}
		R_\ell^z(2\te^\ell) &\leq C_\dun\nu\int_{t_{\ell-1}}^{t_\ell}\lVert \nabla(\bz(t_\ell)- \bz(s))\rVert^2ds + \dun k\lVert \nabla\te^\ell\rVert^2,
		\\
		\label{eq:auxiliary pbm for z error estimate 1 for pressure}
		\int_{t_{\ell-1}}^{t_\ell}(\dv\te^\ell,\pi(s))ds&\leq\frac{k}{\varepsilon}\lVert\dv\te^\ell\rVert^2+\varepsilon\int_{t_{\ell-1}}^{t_\ell}\lVert\pi(s)\rVert^2ds.
		\end{align}
		We add \eqref{eq:auxiliary pbm for z error estimate first2} to \eqref{eq:auxiliary pbm for z error estimate 1 for pressure} with \eqref{eq:auxiliary pbm for z error estimate first1}. Summing up for $\ell=1$ to $\ell= m$,
		\begin{align*}
		\begin{split}
		&\lVert \e^m\rVert^2 + 2\alpha{k}^2\lVert \xiem\rVert^2
		+(\tfrac{\alpha-1}{\alpha})\sum_{\ell = 1}^{m}\lVert\tel-\el \rVert^2
		\\
		&+(2-\dun)\nu\left(k\sum_{\ell = 1}^{m}\lVert \nabla\tel\rVert^2\right)+ \frac{1}{\varepsilon}\EE\left(k\sum_{\ell =1}^{m}\lVert\dv\tel\rVert^2\right)
		\\
		&\leq  C_\dun\nu\sum_{\ell=1}^{m}\int_{t_{\ell-1}}^{t_\ell}\EE\lVert \nabla(\bz(t_\ell)- \bz(s))\rVert^2 +\varepsilon\sum_{\ell=1}^{m}\int_{t_{\ell-1}}^{t_\ell}\EE\lVert\pi(s)\rVert^2ds.
		\end{split}
		\end{align*}
		Using the identity~\eqref{eq:pty ortho proj}, we have
		\begin{align*}
		\begin{split}
		&\frac{1}{\alpha}\lVert \e^m\rVert^2+ \frac{\alpha-1}{\alpha}\lVert \te^m\rVert^2+ 2\alpha{k}^2\lVert \xiem\rVert^2
		+\frac{\alpha-1}{\alpha}\sum_{\ell = 1}^{m-1}\lVert\tel-\el \rVert^2
		\\
		&+(2-\dun)\nu\left(k\sum_{\ell = 1}^{m}\lVert \nabla\tel\rVert^2\right)+ \frac{1}{\varepsilon}\EE\left(k\sum_{\ell =1}^{m}\lVert\dv\tel\rVert^2\right)
		\\
		&\leq  C_\dun\nu\sum_{\ell=1}^{m}\int_{t_{\ell-1}}^{t_\ell}\EE\lVert \nabla(\bz(t_\ell)- \bz(s))\rVert^2 +\varepsilon\sum_{\ell=1}^{m}\int_{t_{\ell-1}}^{t_\ell}\EE\lVert\pi(s)\rVert^2ds.
		\end{split}
		\end{align*}
		Now taking the maximum for $1<m\leq M$, and expectation, we arrive at
		\begin{align}\label{eq:auxiliary pbm for z error estimate first3}
		\begin{split}
		&\EE\max_{1\leq m\leq M}\left\{\frac{1}{\alpha}\lVert \e^m\rVert^2 +\frac{\alpha-1}{\alpha}\lVert \te^m\rVert^2 + 2\alpha{k}^2\lVert \xiem\rVert^2\right\}
		\\
		&+\frac{\alpha-1}{\alpha}\EE\sum_{\ell = 1}^{M-1}\lVert\tel-\el \rVert^2
		\\
		&+(2-\dun)\nu \EE\left(k\sum_{\ell = 1}^{M}\lVert \nabla\tel\rVert^2\right)+ \frac{1}{\varepsilon}\EE\left(k\sum_{\ell =1}^{M}\lVert\dv\tel\rVert^2\right)
		\\
		&\leq  C_\dun\nu\sum_{\ell=1}^{M}\int_{t_{\ell-1}}^{t_\ell}\EE\lVert \nabla(\bz(t_\ell)- \bz(s))\rVert^2 +\varepsilon\sum_{\ell=1}^{M}\int_{t_{\ell-1}}^{t_\ell}\EE\lVert\pi(s)\rVert^2ds.
		\end{split}
		\end{align}
		Finally, we choose $\dun>0$ so that $(2-\dun)$ stays positive and conclude the proof with Lemma~\ref{lem:holder continuity of u}, Proposition~\ref{prop:bound for the pressure} and \eqref{eq:V-estimate of u 2}, and the stability of $\bP_\HH$ in $\WW^{1,2}$.

	\end{proof}
	
	\begin{lem}\label{lem:varpi error estimate}
		Let $\alpha>1$ and $0<\eta<1/2$. For every $\varepsilon>0$, there exists a constant $C = C(T,\nu,\eta)>0$ such that we have
		\begin{align}
		\label{eq:auxiliary pbm for z error estimate 00}
		\EE\left(k\sum_{\ell=1}^{M}\lVert \vpil\rVert^2\right)
		&\leq C (k^\eta+\varepsilon).
		\end{align}
	\end{lem}
	\medskip
	\begin{proof}
		We substitute \eqref{eq:projection error on z} to \eqref{eq:auxilary pbm for z error estimate 3} and arrange the result such that we obtain
		\begin{equation}
		\label{eq:auxiliary pbm for z error estimate 7}
		\begin{split}
		(\e^\ell-\e^{\ell-1},\bfi) &+  \nu k(\nabla\te^\ell, \nabla\bfi)
		\\
		&+ k(\nabla\vpil,\bfi) =R_\ell^\bz(\bfi)+\int_{t_{\ell-1}}^{t_\ell}(\nabla(\pi(t_\ell)-\pi(s)),\bfi)ds.
		\end{split}
		\end{equation}
		Using identity \eqref{eq:pseudo-compressible constraint for z}, we get
		\begin{equation*}
		\begin{split}
		k(\nabla\vpil,\bfi) =(\e^{\ell-1}-\e^\ell,\bfi)+\int_{t_{\ell-1}}^{t_\ell} (\nabla(\pi(t_\ell)-\pi(s)),\bfi)ds  - k\nu(\nabla\te^\ell, \nabla\bfi)+R_\ell^\bz(\bfi).
		\end{split}
		\end{equation*}
		Using inequality \eqref{eq:LBB condition}, we derive that
		\begin{equation}\label{here}
		\begin{split}
		k^2\lVert \vpiel\rVert^2 \leq C\lVert R_\ell^\bz\rVert_{-1}^2&+  {C}\lVert \e^\ell-\e^{\ell-1}\rVert_{-1}^2 +(\nu k)^2\lVert \nabla\el\rVert^2
		\\
		&+ Ck\int_{t_{\ell-1}}^{t_\ell} \lVert\pi(t_\ell)-\pi(s)\rVert^2ds.
		\end{split}
		\end{equation}
		For brevity let us introduce the numbering
		\begin{align*}
		\I + \II + \III &+ \IV
		\\
		&:= C\lVert R_\ell^\bz\rVert_{-1}^2+  {C}\lVert \e^\ell-\e^{\ell-1}\rVert_{-1}^2+ Ck\int_{t_{\ell-1}}^{t_\ell} \lVert\pi(t_\ell)-\pi(s)\rVert^2ds +(\nu k)^2\lVert \nabla\el\rVert^2.
		\end{align*}
		First, we have for $\I$
		\begin{align}\notag
		\I=  \sup_{\bfi\in\WW^{1,2}}\frac{\left(R_\ell^\bz(\bfi)\right)^2}{\lVert \bfi\rVert_1^2}&= \left(\int_{t_{\ell-1}}^{t_\ell}\sup_{\bfi\in\WW^{1,2}}\nu\frac{\left(\nabla (\bz(t_\ell)- \bz(s)),\nabla\bfi\right)}{\lVert \bfi\rVert_1}ds\right)^2
		\\
		&\leq C(\nu,L)k\int_{t_{\ell-1}}^{t_\ell}\lVert\nabla (\bz(t_\ell)- \bz(s)\rVert^2 ds.
		\end{align}
		Now, we estimate the term II. Since $\e^\ell -\e^{\ell-1}\in \cD(\bA^{-1})$, we can take $\bfi=\bA^{-1}(\e^\ell -\e^{\ell-1})$ in identity \eqref{eq:auxiliary pbm for z error estimate 7}. From the orthogonality we get
		\begin{align}\label{eq:i}
		k(\nabla\vpil,\bA^{-1}(\e^\ell -\e^{\ell-1})) &= \int_{t_{\ell-1}}^{t_\ell}(\nabla(\pi(t_\ell)-\pi(s)), \bA^{-1}(\e^\ell -\e^{\ell-1}))ds= 0,
		\\\label{eq:ii}
		k(\nabla\vpil,\bA^{-1}(\e^\ell -\e^{\ell-1})) &= \int_{t_{\ell-1}}^{t_\ell}(\nabla(\pi(t_\ell)-\pi(s)), \bA^{-1}(\e^\ell -\e^{\ell-1}))ds= 0,
		\end{align}
		and from Proposition~\ref{prop:equivalence of norms} we get
		\begin{align}\label{eq:iii}
		\II&\leq C(\e^\ell-\e^{\ell-1},\bA^{-1}(\e^\ell -\e^{\ell-1})).
		\end{align}
		Applying the Young inequality we obtain the following results:
		\begin{align}\label{eq:iv}
		R_\ell^\bz(\bA^{-1}(\e^\ell -\e^{\ell-1}))\leq  C(\dun,\nu) k\int_{t_{\ell-1}}^{t_\ell}\lVert \nabla (\bz(t_\ell)- \bz(s))\rVert^2 ds+\dun\II.
		\end{align}
		Collecting the last four estimates we obtain 
		\begin{align}\label{eq:estimate of e in H-1}
		(1-2\dun)\II\leq C(\dun) k^2\lVert \nabla\te^\ell\rVert^2+ C(\dun,\nu) k\int_{t_{\ell-1}}^{t_\ell}\lVert \nabla (\bz(t_\ell)- \bz(s))\rVert^2 ds.
		\end{align}
		After choosing $\dun$ so that $(1-2\dun)>0$, we substitute  the estimates of $\I$ and $\II$ in \eqref{here}, let the terms $II$ and $IV$ unchanged,  and get
		in this way  the following  new estimate
		\begin{equation}\label{eq:auxiliary pbm for z error estimate 6}
		\begin{split}
		k^2\lVert \vpil\rVert^2&\leq C(\nu,L)k\int_{t_{\ell-1}}^{t_\ell}\lVert\nabla (\bz(t_\ell)- \bz(s)\rVert^2 ds+ C(\nu)k^2\lVert \nabla\te^\ell\rVert^2
		\\
		&+C k\int_{t_{\ell-1}}^{t_\ell}\lVert \nabla (\bz(t_\ell)- \bz(s))\rVert^2 ds
		+ Ck\int_{t_{\ell-1}}^{t_\ell} \lVert\pi(t_\ell)-\pi(s)\rVert^2ds.
		\end{split}
		\end{equation}
		By taking the sum for $\ell=1$ to $\ell=M$ and expectation in \eqref{eq:auxiliary pbm for z error estimate 6}, we get
		\begin{align*}
		k\EE&\left(k\sum_{\ell=1}^{M}\lVert \vpil\rVert^2\right)
		\leq C(\nu,L)k\sum_{\ell = 1}^{M}\int_{t_{\ell-1}}^{t_\ell}\EE\lVert\nabla (\bz(t_\ell)- \bz(s)\rVert^2 ds + +C(\nu) k\EE\left(k\sum_{\ell=1}^{M}\lVert \nabla\te^\ell\rVert^2\right)
		\\
		&+C(\nu) k\sum_{\ell=1}^{M}\int_{t_{\ell-1}}^{t_\ell}\EE\lVert \nabla (\bz(t_\ell)- \bz(s))\rVert^2ds
		+Ck\sum_{\ell=1}^{M}\int_{t_{\ell-1}}^{t_\ell}\EE \lVert\pi(t_\ell)-\pi(s)\rVert^2ds
		.
		\end{align*}
		From Lemma~\ref{lem:holder continuity of u} (iii) and Lemma~\ref{lem:auxiliary pbm for z error estimate} we obtain
		\begin{align*}
		k\EE\left(k\sum_{\ell=1}^{M}\lVert \vpil\rVert^2\right)
		&\leq C(L,T,\nu) (k^{\eta + 1}+k(k^\eta + \varepsilon)).
		\end{align*}

	\end{proof}
	%
	
	
	Let $\bv=\left\lbrace \bv(t,\cdot): t\in[t_{\ell-1},t_\ell]\right\rbrace$ be the strong solution of \eqref{eq:NSE auxiliary 2} as defined in \cref{def:strong solution} and $\rho=\left\lbrace \rho(t,\cdot): t\in[t_{\ell-1},t_\ell]\right\rbrace$ the associated pressure, i.e. for every $t\in[t_{\ell-1},t_\ell]$ and all $\bfi\in\WW^{1,2},\,\chi\in L^2_\per$, we have $\PP\mbox{-a.s.}$
	\begin{align}
	\label{eq:auxiliary 2 strong 1}
	\begin{split}
	(\bv(t_{\ell}),\bfi) + \nu\int_{t_{\ell-1}}^{t_\ell} (\nabla\bv(s),\nabla\bfi)ds+\int_{t_{\ell-1}}^{t_\ell}\tb(\bu(s), \bu(s),\bfi) ds
	\\
	+ \int_{t_{\ell-1}}^{t_\ell} (\nabla\rho(s),\bfi)ds &= \int_{t_{\ell-1}}^{t_\ell}\left(\bfi,d\bW(s)\right),
	\end{split}
	\\
	\label{eq:auxiliary 2 strong 2}
	(\dv\bv(t_{\ell}),\chi) &= 0.
	\end{align}
	To these equations correspond the following algorithm:
	\begin{algo}[Second auxiliary algorithm]\label{algo:auxilary pbm for v}
		Let $\bv^0:=\buo$ be a given $\VV$-valued random variable. Find for every $\ell\in\{1,\ldots,M\}$ a tuple of random variables $(\bvel,\roel)$ with values in $\WW^{1,2}_\per\times L^2_\per$, such that we have $\PP$-a.s.
		
		$\bullet$ Penalization:
		\begin{align}
		\label{eq:auxilary pbm for v 1}
		\begin{split}
		(\tbvel - \bvell,\bfi)+\nu k(\nabla\tbvel,\nabla\bfi)+k\tb(\tbvel+\tbzel,\tbvel+\tbzel,\bfi)
		\\
		+k(\nabla \troel,\bfi )+k(\nabla \psiell,\bfi )&= 0,\;\; \forall\; \bfi\in \WW^{1,2},
		\end{split}
		\\
		\label{eq:auxilary pbm for v 2}
		(\dv\tbvel,\chi) + \varepsilon (\troel,\chi)&=0,\;\;\forall\; \chi \in L^2_\per.
		\end{align}
		
		$\bullet$ Projection:
		\begin{align}
		\label{eq:vprojection 1}
		(\bvel - \tbvel,\bfi) + \alpha k(\nabla(\psiel-\psiell),\bfi)  &=0,\;\forall\; \bfi\in \WW^{1,2}_\per,
		\\
		\label{eq:vprojection 2}
		(\dv \bvel,\chi) &= 0,\;\forall\; \chi \in L^2_\per,
		\\\notag
		\roel &=\troel + \psiel + \alpha(\psiel -\psiell).
		\end{align}
	\end{algo}
	Define the errors $\bsigl= \bv(t_\ell)-\bvel$, $\tbsigl= \bv(t_\ell)-\tbvel$, $\tvrol = \rho(t_\ell)-\troel$, and $\vrol = \rho(t_\ell)-\roel$. Subtracting \eqref{eq:auxilary pbm for v 1} from \eqref{eq:auxiliary 2 strong 1} we get
	\begin{align}\label{eq:auxilary pbm for v error estimate 1}
	\begin{split}
	(\tbsigl-\bsigll,\bfi) +  \nu\int_{t_{\ell-1}}^{t_\ell}(\nabla(\bv(s)-\tbvel),\nabla\bfi)ds+\int_{t_{\ell-1}}^{t_\ell} \tb(\bu(s),\bu(s),\bfi) ds
	\\
	-\int_{t_{\ell-1}}^{t_\ell} \tb(\bu^\ell,\bu^\ell, \bfi) ds+\int_{t_{\ell-1}}^{t_\ell}(\nabla(\rho(s)-\troel),\bfi)ds-k(\nabla\psiell,\bfi) =0.
	\end{split}		
	\end{align}
	Choosing $\chi=\dv \tbvel$ in \eqref{eq:auxilary pbm for v 2} we get
	\begin{equation}
	\label{eq:auxilary pbm for z error estimate 2}
	-(\nabla\troel,\bfi) = \frac{1}{\varepsilon}(\nabla\dv \tbvel,\bfi).
	\end{equation}
	Thanks to the identities
	\begin{align}
	(\nabla(\bv(s)-\tbvel),\nabla\bfi)&=  (\nabla\tbsigl, \nabla\bfi)+ \left(\nabla (\bv(s)- \bv(t_\ell)),\nabla\bfi\right) \mbox{ and }\\
	(\nabla(\rho(s)-\troel),\bfi) &= (\nabla(\rho(s),\bfi) + \frac{1}{\varepsilon}(\nabla\dv\tbvel,\bfi),
	\end{align}
	equation~\eqref{eq:auxilary pbm for v error estimate 1} is reduced to
	\begin{equation}
	\label{eq:auxilary pbm for v error estimate 3}
	\begin{split}
	(\tbsigl-\bsigll,\bfi) &+  \nu k(\nabla\tbsigl, \nabla\bfi)+ \frac{k}{\varepsilon}(\nabla\dv\tbvel,\bfi)
	\\
	& -k(\nabla\psiell,\bfi)=Q_\ell(\bfi) + R_\ell^\bv(\bfi)+\int_{t_{\ell-1}}^{t_\ell}(\dv\tbsigl,\rho(s))ds,
	\end{split}
	\end{equation}
	where
	\[
	\begin{split}
	Q_\ell(\bfi)&=\int_{t_{\ell-1}}^{t_\ell} \Big( \tb(\bu(s),\bu(s),\bfi)-\tb(\tbuel,\tbuel, \bfi)\Big) ds,
	\\
	R_\ell^\bv(\bfi) &= \nu\int_{t_{\ell-1}}^{t_\ell}\left(\nabla (\bv(t_\ell)- \bv(s)),\nabla\bfi\right)ds.
	\end{split}
	\]
	To \eqref{eq:auxilary pbm for v error estimate 3} we associate the following projection equation
	\begin{equation}
	\label{eq:projection error on v}
	\left\{
	\begin{split}
	(\bsigl-\tbsigl,\bfi) &=k\alpha(\nabla(\psiel - \psiell),\bfi),
	\\
	\dv\bsigl &= 0.
	\end{split}\right.
	\end{equation}
	
	Let $\kappa_1, \kappa_2, \kappa_3 >0$ some fixed constants, and let us introduce the sample subsets
	\begin{equation}
	\label{eq:sample subset}
	\begin{split}
	\Omega_{\kappa_1} &= \left\{\omega\in\Omega: \sup_{0\leq t\leq T}\lVert \bu(t)\rVert^2_\VV+k\sum_{\ell=1}^{M}\lVert\bul \rVert_{1}^2 \leq \kappa_1\right\},
	\\
	\Omega_{\kappa_2} &= \left\lbrace\omega\in\Omega: \max_{1\leq m\leq M}\lVert\e^m\rVert^2+\nu k\sum_{\ell=1}^{M}\lVert\el \rVert_{1}^2 +k\sum_{\ell=1}^M \lVert \vpil\rVert^2\leq \kappa_2\right\rbrace, \mbox{ and }
	\\
	\Omega_{\kappa_3} &= \Bigg\{\omega\in\Omega:\forall\;  0\leq s<t\leq T,\;\lVert\bu(s)- \bu(t)\rVert^{2}_{\LL^4} \leq \kappa_3\lvert t-s\rvert^{2\eta}\Bigg\}.
	\end{split}		
	\end{equation}
	In the next paragraph we derive some error estimates on the intersection of these subsets of $\Omega$.
	
	\begin{lem}\label{lem:auxiliary pbm for v error estimate}
		Let $\alpha>1$ and $0<\eta<1/2$. For every $\varepsilon>0$, there exists a constant $C = C(L,T,\nu)>0$ such that
		on $\Omega_{\kappa_1}\cap\Omega_{\kappa_2}\cap\Omega_{\kappa_3}$ we have
		\begin{align}\label{eq:estimate of sigma ell 1}
		\max_{1<m\leq M}\lVert \bsig^m\rVert^2+\nu k\sum_{\ell=1}^{M}\lVert \nabla\bsig^\ell\rVert^2
		\leq C (\kappa_1\kappa_3 k^{2\eta}+\kappa_1\kappa_2 +\kappa_2^2+ k^\eta+\varepsilon)\exp(\kappa_1).
		\end{align}
		
	\end{lem}
	\begin{proof}
		We take $\bfi= 2\tbsigl$ in \eqref{eq:auxilary pbm for v error estimate 3} and proceed exactly like in the proof of Lemma~\ref{lem:auxiliary pbm for z error estimate} until \eqref{eq:auxiliary pbm for z error estimate first3}. Doing so we arrive at
		\begin{equation}\label{eq:auxiliary pbm for v error estimate first3}
		\begin{split}
		\max_{1\leq m\leq M}\left\{(\tfrac{\alpha+1}{2\alpha})\lVert \bsig^m\rVert^2 + (\tfrac{\alpha-1}{2\alpha})\lVert \tbsig^m\rVert^2\right\}
		+(\tfrac{\alpha-1}{2\alpha})\sum_{\ell = 1}^{M-1}\lVert\tbsigl-\bsigl \rVert^2
		\\
		+(\nu-\dun)k\sum_{\ell = 1}^{M}\lVert \nabla\tbsigl\rVert^2
		\leq  C(\dun)\nu\sum_{\ell=1}^{M}\int_{t_{\ell-1}}^{t_\ell}\lVert \nabla(\bv(t_\ell)- \bv(s))\rVert^2
		\\
		+\varepsilon\sum_{\ell=1}^{M}\int_{t_{\ell-1}}^{t_\ell}\lVert\rho(s)\rVert^2ds+2\max_{1\leq m\leq M}\sum_{\ell = 1}^{m}Q_\ell(\tbsigl),
		\end{split}
		\end{equation}
		where
		\[
		\begin{split}
		Q_\ell(\tbsigl)&=\int_{t_{\ell-1}}^{t_\ell} \tb(\bu(s),\bu(s), \tbsigl)- \tb(\tbuel,\tbuel, \tbsigl) ds.
		\end{split}
		\]
		We split the term $Q_\ell$ into four terms as follows
		\begin{align*}
		Q_\ell(\tbsigl) \leq
		\int_{t_{\ell-1}}^{t_\ell}& \Big(\tb(\bu(s),\bu(s)-\bu(t_\ell),\tbsigl) +  \tb(\bu(s)-\bu(t_\ell),\bu(t_\ell),\tbsigl)
		\\
		&+\tb(\bu(t_\ell),\bu(t_\ell)-\tbuel,\tbsigl)+ \tb(\bu(t_\ell)-\tbuel,\tbuel,\tbsigl)\Big) ds
		\\
		\leq \int_{t_{\ell-1}}^{t_\ell} &\left( NLT_1(\tbsigl)+NLT_2(\tbsigl)+NLT_3(\tbsigl)+NLT_4(\tbsigl) \right) ds.
		\end{align*}
		In the next lines, we will estimate the terms $NLT_j(\tbsigl),~j= 1,\ldots,4$, one by one.
		
		$\bullet$ $NLT_1(\tbsigl)$: From \eqref{eq:estimate 1 of tb}, the Sobolev embedding $\WW^{1,2}(D)\subset\LL^4(D)$, and the Young inequality, we get the estimate
		\begin{align*}
		NLT_1(\tbsigl) \leq \lvert \tb(\bu(s),\tbsigl,\bu(s)-\bu(t_\ell)) \rvert
		&\leq C(\dun,L)\lVert\bu(s)\rVert_{1}^2\lVert\bu(s)-\bu(t_\ell)\rVert_{\LL^4}^2+\dun\lVert\tbsigl \rVert_{1}^2.
		\end{align*}
		Then, integrating over the time interval $[t_{\ell-1}, t_{\ell}]$ with respect to $s$, using the H\"older inequality, and since $\omega\in\Omega_{\kappa_1}\cap\Omega_{\kappa_2}\cap\Omega_{\kappa_3}$, we get
		\begin{align*}
		\int_{t_{\ell-1}}^{t_\ell} NLT_1(\tbsigl)ds
		&\leq C(\dun,L)\int_{t_{\ell-1}}^{t_\ell}\lVert \bu(s)\rVert_{1}^2\lVert \bu(s)-\bu(t_\ell)\rVert_{\LL^4}^2 ds+\dun k\lVert\tbsigl \rVert_{1}^2
		\\
		&\leq C(\dun,L)\sup_{t_{\ell-1}\leq s\leq t_\ell}\lVert \bu(s)\rVert_{1}^2\int_{t_{\ell-1}}^{t_\ell}\lVert \bu(s)-\bu(t_\ell)\rVert_{\LL^4}^2 ds+\dun k\lVert\tbsigl \rVert_{1}^2
		\\
		&\leq C(\dun,L)\sup_{t_{\ell-1}\leq s\leq t_\ell}\lVert \bu(s)\rVert_{1}^2\int_{t_{\ell-1}}^{t_\ell}\kappa_3\lvert s-t_\ell\rvert^{2\eta} ds+\dun k\lVert\tbsigl \rVert_{1}^2
		\\[5pt]
		&\leq C(\dun,L)\kappa_1\kappa_3 k^{2\eta+1}+\dun k\lVert\tbsigl \rVert_{1}^2.
		\end{align*}
		
		$\bullet$ $NLT_2(\tbsigl)$: Again from \eqref{eq:estimate 1 of tb} and the Young inequality, we infer
		\begin{align*}
		NLT_2(\tbsigl)\leq \lvert\tb(\bu(s)-\bu(t_\ell),\bu(t_\ell),\tbsigl)\rvert
		&\leq C(\dun) \lVert\bu(s)-\bu(t_\ell) \rVert_{\LL^4}^2\lVert \bu(t_\ell)\rVert_{1}^2 + \dun \lVert \tbsigl\rVert_{1}^2.
		\end{align*}
		Again, integrating over the time interval $[t_{\ell-1}, t_{\ell}]$ with respect to $s$ and since $\omega\in\Omega_{\kappa_2}$, we get
		\begin{align*}
		\int_{t_{\ell-1}}^{t_\ell}  NLT_2(\tbsigl) ds &\leq C(\dun) \lVert \bu(t_\ell)\rVert_{1}^2\int_{t_{\ell-1}}^{t_\ell} \lVert\bu(s)-\bu(t_\ell) \rVert_{\LL^4}^2ds + \dun k\lVert \tbsigl\rVert_{1}^2
		\\
		&\leq C(\dun) \lVert \bu(t_\ell)\rVert_{1}^2\int_{t_{\ell-1}}^{t_\ell}\kappa_3 \lvert  s-t_\ell\rvert^{2\eta } ds + \dun k\lVert \tbsigl\rVert_{1}^2
		\\[5pt]
		&\leq C(\dun) \kappa_3 k^{2\eta + 1}\lVert \bu(t_\ell)\rVert_{1}^2 + \dun k\lVert \tbsigl\rVert_{1}^2.
		\end{align*}
		Summing up from $\ell = 1$ to $\ell= M,$ using the H\"older inequality, and since $\omega\in\Omega_{\kappa_1}$, we get
		\begin{align*}
		\sum_{\ell=1}^{M}\int_{t_{\ell-1}}^{t_\ell}  NLT_2(\tbsigl) ds
		&\leq C(\dun) \kappa_3 k^{2\eta + 1}\sum_{\ell=1}^{M}\lVert \bu(t_\ell)\rVert_{1}^2 + \dun k\sum_{\ell=1}^{M}\lVert \bsig^\ell\rVert_{1}^2
		\\
		&\leq  C(\dun,T) \kappa_1\kappa_3 k^{2\eta } + \dun k\sum_{\ell=1}^{M}\lVert \bsig^\ell\rVert_{1}^2.
		\end{align*}
		
		$\bullet$ $NLT_3(\tbsigl)$: Since $\bu(t_\ell)-\tbuel = \tel+ \tbsigl$ and thanks to the orthogonal property of $\tb$ (see \cref{eq:orthogonal property 1}), we have
		\begin{align*}
		NLT_3(\tbsigl) = \lvert \tb(\bu(t_\ell),\bu(t_\ell)-\tbuel,\tbsigl)\rvert=\lvert  \tb(\bu(t_\ell),\tel+\tbsigl,\tbsigl)\rvert =\lvert \tb(\bu(t_\ell),\tel,\tbsigl)\rvert.
		\end{align*}
		From \eqref{eq:estimate 2 of tb} and the Young inequality, we have
		\begin{align*}
		NLT_3(\tbsigl)
		&\leq C(\dun,L)\lVert \bu(t_\ell)\rVert_{1}^2\lVert \tel\rVert_{1}^2+\dun\lVert\tbsigl \rVert_{1}^2.
		\end{align*}
		As before, integrating over the interval $[t_{\ell-1},t_{\ell}]$ with respect to $s$, we obtain
		\begin{align*}
		\int_{t_{\ell-1}}^{t_\ell} NLT_3(\tbsigl) ds\leq C_\dun(L) k\lVert \bu(t_\ell)\rVert_{1}^2\lVert \tel\rVert_{1}^2+\dun k\lVert\tbsigl \rVert_{1}^2.
		\end{align*}
		Summing up from $\ell = 1$ to $\ell= M$, using the H\"older inequality, and since $\omega\in\Omega_{\kappa_1}\cap\Omega_{\kappa_2}$, we have
		\begin{align*}
		\sum_{\ell=1}^{M}\int_{t_{\ell-1}}^{t_\ell} NLT_3(\bsig^\ell) ds
		&\leq C(\dun,L) k\sum_{\ell=1}^{M}\lVert \bu(t_\ell)\rVert_{1}^2\lVert \tel\rVert_{1}^2+\dun k\sum_{\ell=1}^{M}\lVert\tbsigl \rVert_{1}^2
		\\
		&\leq C(\dun,L) \max_{1\leq\ell\leq M}\lVert \bu(t_\ell)\rVert_{1}^2\left(k\sum_{\ell=1}^{M}\lVert \tel\rVert_{1}^2\right)+\dun k\sum_{\ell=1}^{M}\lVert\tbsigl \rVert_{1}^2
		\\
		&\leq C(\dun,L) \kappa_1\kappa_2+\dun k\sum_{\ell=1}^{M}\lVert\tbsigl \rVert_{1}^2.
		\end{align*}
		
		$\bullet$ $NLT_4(\tbsigl)$: By similar computations as before and using the fact that $\bu(t_\ell)-\tbuel = \tel+ \tbsigl$, we get
		
		\begin{align*}
		NLT_4(\tbsigl)
		=\lvert \tb(\tel+\tbsigl,\tbuel,\tbsigl)\rvert&\leq\lvert \tb(\tel,\tbuel,\tbsigl)\rvert + \lvert \tb(\tbsigl,\tbuel,\tbsigl)\rvert,
		\end{align*}
		For simplicity, let us introduce the notation
		\begin{align}
		NLT_{4,a}(\tbsigl) := \lvert \tb(\tel,\tbuel,\tbsigl)\rvert
		\intertext{and }
		NLT_{4,b}(\tbsigl): = \lvert \tb(\tbsigl,\tbuel,\tbsigl)\rvert.
		\end{align}
		We split  $NLT_{4,a}$ into two terms by replacing $\tbuel$ by $\tbu(t_\ell) + \tel$. Next, we apply \eqref{eq:estimate 2 of tb} and \eqref{eq:estimate 3 of tb} respectively.
		Finally,  we use the Young inequality to get
		\begin{align*}
		NLT_{4,a}(\tbsigl) &\leq\lvert \tb(\tel,\bu(t_\ell),\tbsigl)\rvert+\lvert \tb(\tel,\tel,\tbsigl)\rvert
		\\[5pt]
		&\leq C(\dun,L)\lVert \tel\rVert_{1}^2\lVert\bu(t_\ell) \rVert_{1}^2+C(\dun,L)\lVert \tel\rVert^2\lVert\tel \rVert_{1}^2+\dun\lVert\tbsigl \rVert_{1}^2.
		\end{align*}
		The term $NLT_{4,b}(\tbsigl)$ satisfies the skew-symmetry property (see \eqref{eq:skew-symmetry}). Therefore, using the estimate \eqref{eq:estimate 3 of tb} and the Young inequality, we get
		\begin{align*}
		NLT_{4,b}(\tbsigl) =\lvert \tb(\tbsigl,\tbsigl,\tbuel)\rvert &\leq C(L)\lVert \tbsigl\rVert\lVert \tbsigl\rVert_{1}\lVert \tbuel\rVert_{1}
		\\[5pt]
		&\leq C(\dun,L) \lVert \tbsigl\rVert^2\lVert \tbuel\rVert_{1}^2 + \dun \lVert \tbsigl\rVert_{1}^2.
		\end{align*}
		From these estimates, we obtain after an integration over the time interval $[t_{\ell-1},t_{\ell}]$ with respect to $s$ the estimate
		\begin{align*}
		\int_{t_{\ell-1}}^{t_\ell} NLT_4(\tbsigl) ds \leq C(\dun,L)&k\lVert \tel\rVert_{1}^2\lVert\bu(t_\ell) \rVert_{1}^2+C(\dun,L)k\lVert\tel\rVert^2\lVert\tel \rVert_{1}^2\\&+C(\dun,L)k \lVert \bsig^\ell\rVert^2\lVert \bu^\ell\rVert_{1}^2 + \dun k\lVert \bsig^\ell\rVert_{1}^2.
		\end{align*}
		Then, summing up,
		\begin{align*}
		\sum_{\ell=1}^{M}\int_{t_{\ell-1}}^{t_\ell} NLT_4(\tbsigl) ds &\leq C(\dun,L)k\sum_{\ell=1}^{M} \left[\lVert \tel\rVert_{1}^2\lVert\bu(t_\ell)\rVert_{1}^2+  \lVert\tel\rVert^2\lVert\tel \rVert_{1}^2+  \lVert \tbsigl\rVert^2\lVert \tbul\rVert_{1}^2\right]
		\\
		&+\dun k \sum_{\ell=1}^{M}\lVert\tbsigl \rVert_{1}^2
		\\
		&\leq C(\dun,L)\left[\kappa_1\kappa_2+\kappa_2^2+k\sum_{\ell=1}^{M}\lVert \tbsigl\rVert^2\lVert \tbuel\rVert_{1}^2\right]+\dun k \sum_{\ell=1}^{M}\lVert\tbsigl \rVert_{1}^2.
		\end{align*}
		Finally, the estimates obtained for $NLT_i(\tbsigl),\,i=1\ldots 4$ imply
		\begin{equation}
		\label{eq:nonlinear term Q}
		\begin{split}
		\sum_{\ell=1}^{M}Q_\ell(\tbsigl)\leq C(\dun,L,T)\left[ k\sum_{\ell=1}^{M}\lVert \tbsigl\rVert^2\lVert \tbuel\rVert_{1}^2+(\kappa_1\kappa_3 k^{2\eta}+\kappa_1\kappa_2 +\kappa_2^2)\right] \\+ \dun k\sum_{\ell=1}^{M}\lVert \tbsigl\rVert_{1}^2.
		\end{split}
		\end{equation}
		We plug \eqref{eq:nonlinear term Q} into  \eqref{eq:auxiliary pbm for v error estimate first3}. We fix $\dun$ so that $0<\dun<\nu$. Since $\omega\in \Omega_{\kappa_1}\cap\Omega_{\kappa_2}\cap\Omega_{\kappa_3}$, we can find a constant $C=C(\dun,L,T)>0$ such that
		{
			\begin{align*}
			&\max_{1\leq m\leq M}\left\{(\tfrac{\alpha+1}{2\alpha})\lVert \bsig^m\rVert^2 + (\tfrac{\alpha-1}{2\alpha})\lVert \tbsig^m\rVert^2\right\}
			\\
			&+(\tfrac{\alpha-1}{2\alpha})\sum_{\ell = 1}^{M-1}\lVert\tbsigl-\bsigl \rVert^2 +(\nu-\dun) \left(k\sum_{\ell = 1}^{M}\lVert \nabla\tbsigl\rVert^2\right)
			\\
			&+ \frac{1}{\varepsilon}\left(k\sum_{\ell =1}^{M}\lVert\dv\tbsigl\rVert^2\right)
			\leq C\left[k\sum_{\ell=1}^{M}\lVert \tbsigl\rVert^2\lVert \tbuel\rVert_{1}^2+  (\kappa_1\kappa_3 k^{2\eta}+\kappa_1\kappa_2 +\kappa_2^2+ k^\eta+\varepsilon)\right].
			\end{align*}
			Since we choose  $0<\dun<\nu$, we have $(\nu-\dun)>0$. In addition, since $\omega\in\Omega_{\kappa_1}$, we apply the Gronwall's Lemma we conclude that	
		}
		\begin{align*}
		\max_{1\leq m\leq M}&\left\{(\tfrac{\alpha+1}{2\alpha})\lVert \bsig^m\rVert^2 + (\tfrac{\alpha-1}{2\alpha})\lVert \tbsig^m\rVert^2\right\}
		\\&+k\sum_{\ell=1}^{M}\lVert \nabla\tbsigl\rVert^2
		\leq
		C(\dun,L,T)(\kappa_1\kappa_3 k^{2\eta}+\kappa_1\kappa_2 +\kappa_2^2+ k^\eta+\varepsilon)\exp(\kappa_1).
		\end{align*}
		Remember that $\bP_\HH$ is stable in $\WW^{1,2}$, thus $\lVert \nabla\bsigl\rVert\leq C\lVert \nabla\tbsigl\rVert$.

	\end{proof}
	
	\medskip
	
	\begin{lem}\label{lem:estimation of rho} Under the same assumption as in Lemma~\ref{lem:auxiliary pbm for v error estimate}, there exists a constant $C = C(L,T,\nu)>0$ such that on $\Omega_{\kappa_1} \cap\Omega_{\kappa_2} \cap\Omega_{\kappa_3} $ the error iterates $\{\vrol:1\leq \ell\leq M\}$ of the pressure term in \cref{algo:auxilary pbm for v} satisfies
		\begin{equation}
		\label{eq:auxiliary pbm for v error estimate 00}
		k\sum_{\ell= 1}^M\lVert \vrol\rVert^2
		\leq  C(\kappa_1\kappa_3 k^{2\eta}+\kappa_1\kappa_2 +\kappa_2^2+ k^\eta+\varepsilon)\exp(\kappa_1).
		\end{equation}
		
	\end{lem}
	\begin{proof}
		We add \eqref{eq:projection error on v} and \eqref{eq:auxilary pbm for v error estimate 3} and get
		\begin{equation}
		\label{eq:estimation of rho1}
		\begin{split}
		k(\nabla\vrol,\bfi) =Q_\ell(\bfi)&+R_\ell^\bv(\bfi)+\int_{t_{\ell-1}}^{t_\ell} (\nabla(\rho(t_\ell)-\rho(s)),\bfi)ds \\&-(\bsigl-\bsigll,\bfi) - k(\nabla\tbsigl, \nabla\bfi).
		\end{split}
		\end{equation}
		Using the inequality \eqref{eq:LBB condition}, we derive that
		\begin{align*}
		k^2\sum_{\ell=0}^{M}\lVert \vrol\rVert^2&\leq C\sum_{\ell=0}^{M} \sup_{\bfi\in\WW^{1,2}} \frac{1}{\lVert \bfi\rVert_1^2}[Q_\ell(\bfi)+R_\ell(\bfi)+\int_{t_{\ell-1}}^{t_\ell} (\nabla(\rho(t_\ell)-\rho(s)),\bfi)ds
		\\
		&\qquad\qquad\qquad\qquad\qquad\qquad\qquad -(\bsig^\ell-\bsig^{\ell-1},\bfi) - k(\nabla\bsig^\ell, \nabla\bfi)]^2.
		\end{align*}
		For simplicity let us introduce the following abbreviation
		\begin{align*}
		\tI &+ \tII + \tIII + \tIV + \tV
		\\
		&:= \frac{1}{\lVert \bfi\rVert_1}Q_\ell(\bfi)+R_\ell(\bfi)+\int_{t_{\ell-1}}^{t_\ell} (\nabla(\rho(t_\ell)-\rho(s)),\bfi)ds -(\bsig^\ell-\bsig^{\ell-1},\bfi) - k(\nabla\bsig^\ell, \nabla\bfi).
		\end{align*}
		In the following we estimate each term of the right side.
		
		$\bullet$ \textbf{Term} $\tI$: Here, we get
		\begin{align*}
		\tI&\leq C \int_{t_{\ell-1}}^{t_\ell}\sup_{\bfi\in\WW^{1,2}}\frac{1}{\lVert \bfi\rVert_1}(NLT_1(\bfi)+NLT_2(\bfi)+NLT_3(\bfi)+NLT_4(\bfi))ds,\\
		&\leq  C \int_{t_{\ell-1}}^{t_\ell} (\tNLT_1+\tNLT_2+\tNLT_3+\tNLT_4) ds,
		\end{align*}
		where with \eqref{eq:estimate 1 of tb} and \eqref{eq:estimate 2 of tb} we arrive at
		\begin{align*}
		\tNLT_1&\leq C(L)\sup_{\bfi\in\WW^{1,2}} \frac{1}{\lVert \bfi\rVert_1}
		\lVert\bu(s)\rVert_{1}\lVert\bfi\rVert_{1}\lVert \bu(s)-\bu(t_\ell)\rVert_{\LL^4}=\lVert\bu(s)\rVert_{1}\lVert \bu(s)-\bu(t_\ell)\rVert_{\LL^4},\\
		\tNLT_2&\leq \sup_{\bfi\in\WW^{1,2}} \frac{1}{\lVert \bfi\rVert_1} \lVert\bu(s)-\bu(t_\ell) \rVert_{\LL^4}\lVert \bu(t_\ell)\rVert_{1}\lVert \bfi\rVert_{1}=\lVert\bu(s)-\bu(t_\ell) \rVert_{\LL^4}\lVert \bu(t_\ell)\rVert_{1},\\
		\tNLT_3&\leq C(L)\sup_{\bfi\in\WW^{1,2}} \frac{1}{\lVert \bfi\rVert_1}\lVert \bu(t_\ell)\rVert_{1}\lVert \tel\rVert_{1}\lVert\bfi \rVert_{1}=\lVert \bu(t_\ell)\rVert_{1}\lVert \tel\rVert_{1},\\
		\tNLT_4&\leq C(L)\sup_{\bfi\in\WW^{1,2}} \frac{1}{\lVert \bfi\rVert_1}\{\lVert \tel\rVert_{1}\lVert\bu(t_\ell) \rVert_{1}\lVert\bfi \rVert_{1}+\lVert \tel\rVert\lVert\tel \rVert_{1}\lVert\bfi\rVert_{1}\},
		\\
		&=C(L)\{\lVert \tel\rVert_{1}\lVert\bu(t_\ell) \rVert_{1}+\lVert \tel\rVert\lVert\tel \rVert_{1}\}.
		\end{align*}
		Integrating gives
		\begin{align*}
		\int_{t_{\ell-1}}^{t_\ell}\tNLT_1 ds
		&\leq C(L)\sup_{t_{\ell-1}\leq s\leq t_\ell}\lVert\bu(s)\rVert_{1}\int_{t_{\ell-1}}^{t_\ell}\lVert \bu(s)-\bu(t_\ell)\rVert_{\LL^4}ds,
		\\
		\int_{t_{\ell-1}}^{t_\ell}\tNLT_2 ds&\leq\lVert \bu(t_\ell)\rVert_{1}\int_{t_{\ell-1}}^{t_\ell}\lVert\bu(s)-\bu(t_\ell) \rVert_{\LL^4}ds,
		\\
		\int_{t_{\ell-1}}^{t_\ell}\tNLT_3 ds&\leq k\lVert \bu(t_\ell)\rVert_{1}\lVert \tel\rVert_{1},
		\\
		\int_{t_{\ell-1}}^{t_\ell}\tNLT_4 ds&\leq C(L)k\{\lVert \tel\rVert_{1}\lVert\bu(t_\ell) \rVert_{1}+\lVert \tel\rVert\lVert\tel \rVert_{1}\}.
		\end{align*}
		From the estimates of $\int_{t_{\ell-1}}^{t_\ell}\tNLT_ids,$ for $\,i=1,\ldots,4,$ we obtain
		\begin{align*}
		\tI^2\leq k C(L)&\sup_{t_{\ell-1}\leq s\leq t_\ell}\lVert\bu(s)\rVert_{1}^2\int_{t_{\ell-1}}^{t_\ell}\lVert \bu(s)-\bu(t_\ell)\rVert_{\LL^4}^2ds+k\lVert \bu(t_\ell)\rVert_{1}^2\int_{t_{\ell-1}}^{t_\ell}\lVert\bu(s)-\bu(t_\ell) \rVert_{\LL^4}^2ds
		\\[5pt]
		&+k^2\lVert \bu(t_\ell)\rVert_{1}^2\lVert \tel\rVert_{1}^2+C(L)k^2\{\lVert \tel\rVert_{1}^2\lVert\bu(t_\ell) \rVert_{1}^2+\lVert \tel\rVert^2\lVert\tel \rVert_{1}^2\}.
		\end{align*}
		Summing up for $\ell = 1$ to $\ell=M$ gives
		\begin{align*}
		\sum_{\ell=1}^{M}\tI^2&\leq
		C(L,T)k(\kappa_1\kappa_2+\kappa_1\kappa_3+\kappa_2^2).
		\end{align*}

		$\bullet$ \textbf{Term} $\tII$: Here, we have
		\begin{align*}
		\tII^2
		&\leq k\int_{t_{\ell-1}}^{t_\ell}\sup_{\bfi\in\WW^{1,2}}\nu^2\frac{\lVert\nabla (\bv(t_\ell)- \bv(s)\rVert^2\lVert\nabla\bfi\rVert^2}{\lVert \bfi\rVert_1^2}ds\leq C k\int_{t_{\ell-1}}^{t_\ell}\nu^2\lVert\nabla (\bv(t_\ell)- \bv(s)\rVert^2 ds.
		\intertext{Then, summing up and using Lemma~\ref{lem:holder continuity of u} gives}
		\sum_{\ell=1}^M\tII^2&\leq C(\nu,T)\leq C(\nu) k^{\eta + 1}.
		\end{align*}
		
		$\bullet$ \textbf{Term} $\tIII$: Here, we have
		\begin{align*}
		\tIII^2 &\leq \sup_{\bfi\in\WW^{1,2}} \frac{1}{\lVert \bfi\rVert_1^2}k\int_{t_{\ell-1}}^{t_\ell} \lVert\rho(t_\ell)-\rho(s)\rVert^2\lVert\bfi\rVert_1^2ds=k\int_{t_{\ell-1}}^{t_\ell}\lVert\rho(t_\ell)-\rho(s)\rVert^2 ds.
		\intertext{Again summing up and using Lemma~\ref{lem:holder continuity of u} gives}
		\sum_{\ell=1}^M\tIII^2 &\leq k\sum_{\ell=1}^M\int_{t_{\ell-1}}^{t_\ell} \lVert\rho(t_\ell)-\rho(s)\rVert^2 ds\leq C_{T,4}k\sum_{\ell=1}^Mk^{2\eta+1} = C_{T,4}k^{2\eta+1}.
		\end{align*}
		
		$\bullet$ \textbf{Term} $\tIV$: Here, we proceed in two steps. First, we estimate $\tIV$ with a term under a weak  norm. Then, we use the Proposition~\ref{prop:equivalence of norms} to bound this later with terms under $H^1$ or $L^2$-norm. Thereby, we have
		\begin{align*}
		\tIV=\sup_{\bfi\in\WW^{1,2}} \frac{1}{\lVert \bfi\rVert_1}(\bsig^\ell-\bsigll,\bfi)\leq \sup_{\bfi\in\WW^{1,2}} \frac{1}{\lVert \bfi\rVert_1}\lVert \bsigl-\bsigll\rVert_{-1}\lVert \bfi\rVert_1\leq\lVert \bsigl-\bsigll\rVert_{-1}.
		\end{align*}
		Next, since $\bsigl -\bsigll\in \cD(\bA^{-1})$, we can take $\bfi=\bA^{-1}(\bsigl -\bsigll)$ in \eqref{eq:estimation of rho1}, use Proposition~\ref{prop:equivalence of norms}, and arrive at the following estimates:
		\\[1pt]
		\begin{itemize}
			\item[$i)$] $\lVert \bsigl-\bsigll\rVert_{{-1}}^2 \leq C(\bsigl-\bsigll,\bA^{-1}(\bsigl -\bsigll)),$
			\\[1pt]
			\item[$ii)$] $k(\nabla\tbsigl, \nabla \bA^{-1}(\bsigl -\bsigll))\leq \dun\lVert \bsigl-\bsigll\rVert_{{-1}}^2 + C\dun k^2\lVert \nabla\tbsigl\rVert^2,$
			\\[1pt]
			\item[$iii)$] $k(\nabla\vrol,\bA^{-1}(\bsigl -\bsigll)) = \int_{t_{\ell-1}}^{t_\ell}(\nabla(\rho(t_\ell)-\rho(s)), \bA^{-1}(\bsig^\ell -\bsig^{\ell-1}))ds= 0,$
			\\[1pt]
			\item[$iv)$] $R_\ell^\bv(\bA^{-1}(\bsigl -\bsigll))\leq C(\nu,\dun) k^{\eta +2} + {\dun}\lVert \bsigl -\bsigll\rVert_{-1}^2.$
			\\[1pt]
		\end{itemize}
		We split the term $Q_\ell(\bA^{-1}(\bsigl -\bsigll))$ as follows
		\begin{align*}
		Q_\ell(\bA^{-1}(\bsigl -\bsigll))
		\leq \int_{t_{\ell-1}}^{t_\ell} & NLT_1(\bA^{-1}(\bsigl -\bsigll))+NLT_2(\bA^{-1}(\bsigl -\bsigll))\\&+NLT_3(\bA^{-1}(\bsigl -\bsigll))+NLT_4(\bA^{-1}(\bsigl -\bsigll))ds,
		\end{align*}
		where each of terms $NLT_j(\bA^{-1}(\bsigl -\bsigll))$ for $j=1,2,3,4$, are estimated as follows:
		\begin{align*}
		NLT_1(\bA^{-1}(\bsigl -\bsigll))
		&\leq C_\dun(L)k\lVert\bu(s)\rVert_{1}^2\lVert\bu(s)-\bu(t_\ell)\rVert_{\LL^4}^2+\frac{\dun}{k}\lVert\bA^{-1}(\bsigl -\bsigll) \rVert_{1}^2,
		\\[7pt]
		NLT_2(\bA^{-1}(\bsigl -\bsigll))
		&\leq C_\dun k \lVert\bu(s)-\bu(t_\ell) \rVert_{\LL^4}^2\lVert \bu(t_\ell)\rVert_{1}^2 + \frac{\dun}{k} \lVert \bA^{-1}(\bsigl -\bsigll)\rVert_{1}^2,
		\\[7pt]
		NLT_3(\bA^{-1}(\bsigl -\bsigll))
		&\leq C(\dun,L)k\lVert \bu(t_\ell)\rVert_{1}^2\lVert \tel\rVert_{1}^2+\frac{\dun}{k}\lVert\bA^{-1}(\bsigl -\bsigll) \rVert_{1}^2,
		\\[7pt]
		NLT_4(\bA^{-1}(\bsigl -\bsigll))
		&\leq C(\dun,L)k\left\{\lVert \tel\rVert_{1}^2\lVert\bu(t_\ell) \rVert_{1}^2+\lVert \tel\rVert^2\lVert\tel \rVert_{1}^2+\lVert \tbsigl\rVert^2\lVert \tbuel\rVert_{1}^2\right\}
		\\
		&\hspace{150pt} + \frac{2\dun}{k} \lVert \bA^{-1}(\bsigl -\bsigll)\rVert_{1}^2.
		\end{align*}
		All together, the estimates of $NLT_i(\bA^{-1}(\bsigl -\bsigll))$, for $i= 1,\ldots,4$, lead to
		\begin{align*}
		Q_\ell(\bA^{-1}(\bsigl -\bsigll))
		&\leq C{(\dun,L)}k\int_{t_{\ell-1}}^{t_\ell}\lVert\bu(s)-\bu(t_\ell)\rVert_{\LL^4}^2\left[\lVert\bu(s)\rVert_{1}^2+  \lVert \bu(t_\ell)\rVert_{1}^2\right]ds
		\\[2pt]
		&+C(\dun,L)k^2\left\{\lVert \tel\rVert_{1}^2\lVert\bu(t_\ell) \rVert_{1}^2+\lVert \tel\rVert^2\lVert\tel \rVert_{1}^2+\lVert \tbsigl\rVert^2\lVert \tbuel\rVert_{1}^2\right\}
		\\[5pt]
		&+C(\dun,L)k^2\lVert \bu(t_\ell)\rVert_{1}^2\lVert \tel\rVert_{1}^2+{4\dun}\lVert\bA^{-1}(\bsig^\ell -\bsig^{\ell-1}) \rVert_{1}^2.
		\end{align*}
		In addition on $\Omega_{\kappa_3}$ we have
		\begin{align*}
		Q_\ell(\bA^{-1}(\bsigl -\bsigll))
		\leq &C(\dun,L)(\sup_{t_{\ell-1}\leq s\leq {t_\ell}}\lVert\bu(s)\rVert_{1}^2 +C(\dun)\lVert \bu(t_\ell)\rVert_{1}^2)\kappa_3k^{2\eta + 2}
		\\[3pt]
		&+C(\dun,L)k^2\left\{\lVert \tel\rVert_{1}^2\lVert\bu(t_\ell) \rVert_{1}^2+\lVert \tel\rVert^2\lVert\tel \rVert_{1}^2+\lVert \tbsigl\rVert^2\lVert \tbuel\rVert_{1}^2\right\}
		\\[5pt]
		&+C(\dun,L)k^2\lVert \bu(t_\ell)\rVert_{1}^2\lVert \tel\rVert_{1}^2
		+{4\dun}\lVert\bsigl -\bsigll \rVert_{-1}^2.
		\end{align*}
		Now summing for $\ell=1$ to $\ell=M$, we have
		\begin{align*}
		\sum_{\ell=1}^{M}Q_\ell(\bA^{-1}&(\bsigl -\bsigll))
		\leq C(\dun,L)(\sup_{0\leq s\leq T}\lVert\bu(s)\rVert_{1}^2 +\max_{1\leq\ell\leq M} \lVert \bu(t_\ell)\rVert_{1}^2)\kappa_3 k^{2\eta + 2}
		\\&+C(\dun,L)k\max_{1\leq\ell\leq M}\lVert \bu(t_\ell)\rVert_{1}^2\left(k\sum_{\ell=1}^{M}\lVert \tel\rVert_{1}^2\right)
		+k\max_{1\leq\ell\leq M}\lVert \tel\rVert^2\left( k\sum_{\ell=1}^{M}\lVert\tel \rVert_{1}^2\right)\\
		&+k\max_{1\leq\ell\leq M}\lVert \tbsigl\rVert^2\left(k\sum_{\ell=1}^{M}\lVert \tbuel\rVert_{1}^2\right)
		+{4\dun}\sum_{\ell=1}^{M}\lVert\bsigl -\bsigll \rVert_{-1}^2.
		\end{align*}
		Since we have due to the assumptions $\omega \in \Omega_{\kappa_1}\cap\Omega_{\kappa_2}\cap\Omega_{\kappa_3}$
		we obtain using \eqref{eq:estimate of sigma ell 1} 
		\begin{align*}
		\sum_{\ell=1}^{M}Q_\ell(\bA^{-1}(\bsigl -\bsigll))
		&\leq C(\dun,L)k\left(\kappa_1\kappa_3 k^{2\eta + 1}+\kappa_1\kappa_2+\kappa_2^2\right)+{4\dun}\sum_{\ell=1}^{M}\lVert\bsigl -\bsigll \rVert_{-1}^2\\
		&+	C(\dun,L,T)k(\kappa_1\kappa_3 k^{2\eta}+\kappa_1\kappa_2 +\kappa_2^2+ k^\eta+\varepsilon)\exp(\kappa_1).
		\end{align*}
		All together we obtain,
		\begin{align*}
		\sum_{\ell=1}^{M}\lVert \bsigl-\bsigll\rVert_{-1}^2&\leq 6\dun\sum_{\ell=1}^{M}\lVert \bsigl-\bsigll\rVert_{{-1}}^2 + C(\dun )k^2\sum_{\ell=1}^{M}\lVert \nabla\tbsigl\rVert^2+C(\nu,\dun) k^{\eta +1}
		\\
		&+C(\dun,L,T)k(\kappa_1\kappa_3 k^{2\eta}+\kappa_1\kappa_2 +\kappa_2^2+ k^\eta+\varepsilon)\exp(\kappa_1).
		\end{align*}
		The terms with $\lVert \bsig^\ell-\bsig^{\ell-1}\rVert_{-1}^2$ are absorbed by the left hand side. Thanks to \eqref{eq:estimate of sigma ell 1},
		\begin{align*}
		(1-6\dun)\sum_{\ell=1}^{M}\lVert \bsigl-\bsigll\rVert_{-1}^2&\leq C(\dun,L,T)k(\kappa_1\kappa_3 k^{2\eta}+\kappa_1\kappa_2 +\kappa_2^2+ k^\eta+\varepsilon)\exp(\kappa_1).
		\end{align*}
		We can choose $\dun$ so that $(1-6\dun)>0$. Note that $1\leq\exp(x)$ for all $x\in\RR$. Therefore,
		\begin{align*}
		\sum_{\ell=1}^{M}{\tIV}^2&\leq C(L,T)k(\kappa_1\kappa_3 k^{2\eta}+\kappa_1\kappa_2 +\kappa_2^2+ k^\eta+\varepsilon)\exp(\kappa_1).
		\end{align*}
		$\bullet$ \textbf{Term} $\tV$: Here, we have
		\begin{align*}
		\tV	&= \sup_{\bfi_\ell\in\WW^{1,2}} \frac{1}{\lVert \bfi\rVert_1}k(\nabla\tbsigl, \nabla\bfi)
		\leq \sup_{\bfi\in\WW^{1,2}} \frac{1}{\lVert \bfi\rVert_1}k\lVert \nabla\tbsigl\rVert\lVert\nabla\bfi \rVert=Ck\lVert \nabla\tbsigl\rVert.
		\intertext{Summing up and using \eqref{eq:estimate of sigma ell 1} gives}
		\sum_{\ell=1}^{M}\tV^2&\leq Ck^2\sum_{\ell=1}^{M}\lVert \nabla\tbsigl\rVert^2\leq C(\dun,L,T)k(\kappa_1\kappa_3 k^{2\eta}+\kappa_1\kappa_2 +\kappa_2^2+ k^\eta+\varepsilon)\exp(\kappa_1).
		\end{align*}
		Collecting $\tI,\II,\tIII,\tIV,$ and $\tV$, we obtain
		\begin{align*}
		k^2\sum_{\ell=0}^{M}\lVert \vro^\ell\rVert^2&\leq \sum_{\ell=0}^{M}\{\tI + \tII + \tIII + \tIV + \tV\}^2,\\
		&\leq C(L,T)k(\kappa_1\kappa_2+\kappa_1\kappa_3+\kappa_2^2)+C(\nu,T) k^{\eta + 1} 
		+ C_{T,4}k^{2\eta+1}\\
		&\qquad+C(L,T)k(\kappa_1\kappa_3 k^{2\eta}+\kappa_1\kappa_2 +\kappa_2^2+ k^\eta+\varepsilon)\exp(\kappa_1)\\
		&\qquad+C(L,T)k(\kappa_1\kappa_3 k^{2\eta}+\kappa_1\kappa_2 +\kappa_2^2+ k^\eta+\varepsilon)\exp(\kappa_1).
		\end{align*}
		Because $1<\exp(x)$ for all $x\in\RR$ and with a limiting order term ($k^\eta$), we have
		\begin{align*}
		k\sum_{\ell=0}^{M}\lVert \vrol\rVert^2
		\leq C(L,T,\nu)(\kappa_1\kappa_3 k^{2\eta}+\kappa_1\kappa_2 +\kappa_2^2+ k^\eta+\varepsilon)\exp(\kappa_1).
		\end{align*}
		%
		
	\end{proof}
	\section{Main results}
	\label{sec:main results}Let us define the errors $\beel= \bu(t_\ell)-\buel$ and $\qrel = \pre(t_\ell)-\preel$. Here in the final section, we use the estimates of the iterates $\{\el,\vpil\}_\ell$ and $\{\bsigl,\vrol\}_\ell$ to derive an estimate for $\{\beel,\qrel \}_\ell$, show convergence in probability of \cref{algo:penalty method}, and deduce from that strong convergence.
	
	We set
	\begin{align}
	\label{eq:EM}
	\mE^M &:= \max_{1\leq m\leq M}\lVert \bee^m\rVert^2 + \nu k\sum_{\ell=1}^M \lVert \nabla\bee^\ell\rVert^2+k\sum_{\ell=1}^M \lVert \qre^\ell\rVert^2,
	\\
	\label{eq:tEM}
	\tmE^M &:= \max_{1\leq m\leq M}\lVert \bee^m\rVert^2 + \left(\nu k\sum_{\ell=1}^M \lVert \nabla\bee^\ell\rVert^2\right)^{1/2}+\left(k\sum_{\ell=1}^M \lVert \qre^\ell\rVert^2\right)^{1/2},
	\\
	\label{eq:E1M}
	\mE_1^M &:= \max_{1\leq m\leq M}\lVert \e^m\rVert^2 + \nu k\sum_{\ell=1}^M \lVert \nabla\el\rVert^2+k\sum_{\ell=1}^M \lVert \vpil\rVert^2,
	\\
	\label{eq:E2M}
	\mE_2^M &:= \max_{1\leq m\leq M}\lVert \bsig^m\rVert^2 + \nu k\sum_{\ell=1}^M \lVert \nabla\bsigl\rVert^2+k\sum_{\ell=1}^M \lVert \vrol\rVert^2.
	\end{align}
	\begin{theorem}\label{thm:convergence in probability}
		Let  $\mE^M$ be defined in \eqref{eq:EM}. If $\varepsilon=\kappa^\eta$, the \cref{algo:penalty method} converges in probability with order $0<r< \eta$. In particular, we have
		\begin{equation*}
		\lim_{\tC\to\infty}\lim_{k\to 0}\PP\left[\mE^M \geq \tC  k^\mathrm{r}\right]=0.
		\end{equation*}
	\end{theorem}
	\begin{proof}Let $\tC,r>0$ be some arbitrary constants which will be fixed at the end of the proof. By the Chebyshev inequality
		\begin{align*}
		&\PP \left[\mE^M \geq \tC k^r\right]\leq
		\PP(\Omega \setminus\Omega_{\kappa_1})+\PP(\Omega \setminus\Omega_{\kappa_2})+\PP(\Omega \setminus\Omega_{\kappa_3})+\PP\left[\mE^M \geq \tC k^r\big|\Omega_{\kappa_1}\cap\Omega_{\kappa_2}\cap\Omega_{\kappa_3}\right]
		\\
		&\leq \frac{1}{\kappa_1}{\EE\left[\displaystyle\sup_{0\leq s\leq T}\lVert \bu(s)\rVert^2_\VV+\nu k\sum_{\ell=1}^{M}\lVert\bu^\ell \rVert_{1}^2\right]}+\frac{1}{{\kappa_2}}{\EE\left[\displaystyle\max_{1\leq \ell\leq M}\lVert \el\rVert^2+\nu k\sum_{\ell=1}^{M}\lVert\el \rVert_{1}^2+k\sum_{\ell=1}^M \lVert \vpil\rVert^2\right]}
		\\
		&\qquad+\frac{1}{{\kappa_3\lvert t-s\rvert^{2\eta}}}{ \EE\left[\lVert \bu(s)- \bu(t)\rVert^{2}_{\LL^4}\right]}+\frac{\EE\left[\mE^M \big|\Omega_{\kappa_1}\cap\Omega_{\kappa_2}\cap\Omega_{\kappa_3}\right]}{\tC k^r}.
		\end{align*}
		Observe, that we can write $\bee^\ell = \el + \bsigl$ and $\qrel = \vpil + \vrol$. Now, it follows by the definition of $\Omega_{\kappa_2}$ (see \eqref{eq:sample subset}), by Lemma~\ref{lem:auxiliary pbm for v error estimate}, and Lemma~\ref{lem:estimation of rho} 
		\begin{align*}
		\EE\left[\mE^M \big| \Omega_{\kappa_1}\cap\Omega_{\kappa_2}\cap\Omega_{\kappa_3}\right]&\leq \EE\left[\mE_1^M \big| \Omega_{\kappa_1}\cap\Omega_{\kappa_2}\cap\Omega_{\kappa_3}\right] + \EE\left[\mE_2^M \big| \Omega_{\kappa_1}\cap\Omega_{\kappa_2}\cap\Omega_{\kappa_3}\right]
		\\
		&\leq \kappa_2 + C(\kappa_1\kappa_3 k^{2\eta}+\kappa_1\kappa_2 +\kappa_2^2+ k^\eta+\varepsilon)\exp(\kappa_1),
		\end{align*}
		where $\mE_1^M$ and $\mE_2^M$ are defined by \eqref{eq:E1M} and \eqref{eq:E2M} respectively.
		%
		Moreover, by estimate \eqref{eq:V-estimate of u 2}, Lemma~\ref{lem:stability}, and Lemma~\ref{lem:auxiliary pbm for z error estimate} we obtain
		\begin{align*}
		\PP \left[\mE^M\geq \tC k^r\right]&\leq \frac{\kappa_2 + C(\kappa_1\kappa_3 k^{2\eta}+\kappa_1\kappa_2 +\kappa_2^2+ k^\eta+\varepsilon)\exp(\kappa_1)}{\tC k^r}
		+\frac{C}{\kappa_1}+\frac{C(k^\eta +\varepsilon)}{\kappa_2}+\frac{C}{\kappa_3}
		\\
		&\leq \frac{C(\kappa_2 + \kappa_3 k^{2\eta} +\kappa_2^2+ k^\eta+\varepsilon)\exp(2\kappa_1)}{\tC k^r}+\frac{C}{\kappa_1}+\frac{C(k^\eta +\varepsilon)}{\kappa_2}+\frac{C}{\kappa_3}.
		\end{align*}
		Let $\mu>0$. We fix $\varepsilon = k^\eta$, $\kappa_1=\ln k^{-\mu/2}$, $\kappa_2=k^{\mu + r}$, and $\kappa_3= k^{-\eta}$ with $k<1$. Therefore, we have
		\begin{align*}
		\PP\left[\mE^M\geq \tC k^r\right]
		&\leq\frac{C(k^{r}+k^{\eta-\mu})}{\tC k^r}-\frac{C}{\ln k^{\mu}}+C{k^{\eta-\mu - r}}+C k^{\eta}.
		\end{align*}
		Let us remind, that we fixed the constant $r$ in the beginning, such that $\eta -\mu -r>0$. Now, we are ready to go to the limit:
		\begin{align*}
		\lim_{\tC\to \infty}\lim_{k\to 0}\PP\left[\mE^M\geq \tC k^r\right]
		\leq\lim_{\tC\to \infty}\lim_{k\to 0}\left(\frac{C}{\tC}-\frac{C}{\ln k^{\mu}}+C{k^{\eta-\mu -r }} + Ck^\eta\right)=\lim_{\tC\to \infty} \frac{C}{\tC} = 0.
		\end{align*}
		This gives the assertion.
		
	\end{proof}
	
	A consequence of this theorem is strong convergence of iterates of the scheme. This will be shown by the following corollary.
	
	\begin{cor}Let  $\tmE^M$ be defined as in \eqref{eq:tEM}.
		Under the assumption of \cref{thm:convergence in probability} we have
		\begin{align*}
		\lim_{M\to \infty}\EE\left[\;\tmE^M\,\right] &=0.
		\end{align*}
		
	\end{cor}
	\begin{proof} Let $\tC>0$ an arbitrary constant.
		We define the sample set
		\begin{equation*}
		\Omega_{\tC,k}:=\left\{ \mE^M\geq \tC  k^\mathrm{r}\right\}.
		\end{equation*}
		%
		From the law of total probability we deduce that
		\begin{align*}
		\EE\left[\;\tmE^M\,\right] = \EE\left[\;\tmE^M\,\Big| \Omega_{\tC,k} \right]\PP(\Omega_{\tC,k}) + \EE\left[\;\tmE^M\,\Big|\Omega\setminus\Omega_{\tC,k}\right]\PP(\Omega\setminus\Omega_{\tC,k}).
		\end{align*}
		Since $\PP(\Omega\setminus\Omega_{\tC,k})\leq 1$, and by definition of $\Omega_{\tC,k}$,
		\begin{equation*}
		\EE\left[\;\tmE^M\,\right]\leq
		\EE\left[\;\tmE^M\,\Big| \Omega_{\tC,k} \right]\PP(\Omega_{\tC,k}) + \tC k^{r/2}.
		\end{equation*}
		Using the definition of conditional expectation and the Cauchy--Schwartz inequality we obtain
		\begin{align*}
		\EE\left[\;\tmE^M\,\Big| \Omega_{\tC,k} \right]\PP(\Omega_{\tC,k})
		\leq \EE\left[\;\left(\tmE^M\right)^2\,\right]\left(\PP(\Omega_{\tC,k})\right)^{1/2}.
		\end{align*}
		Remember that $\beel = \bu(t_\ell) -\buel$ and $\beel = \pre(t_\ell) -\preel$. Using now \cref{eq:L2-estimate of u}, Lemma~\ref{lem:stability}(iii), \cref{eq:V-estimate of u 1}, Proposition~\ref{prop:bound for the pressure}, and Lemma~\ref{lem:stability for pressure}, we arrive at
		\begin{align*}
		\EE&\left[\;\left(\tmE^M\right)^2\,\right]\leq  \EE\left[\max_{1\leq m\leq M}\lVert \bu(t_m)\rVert^4\right] +  \EE\left[\max_{1\leq m\leq M}\lVert \buem\rVert^4\right]+\EE\left(\nu k\sum_{\ell= 1}^{M}\lVert \nabla\bu(t_\ell)\rVert^2\right)
		\\
		&+ \EE\left(\nu k\sum_{\ell= 1}^{M}\lVert \nabla\buel\rVert^2\right) +\EE\left( k\sum_{\ell= 1}^{M}\lVert \pre(t_\ell)\rVert^2\right) + \EE\left( k\sum_{\ell= 1}^{M}\lVert \preel\rVert^2\right)\leq C(T,L,\bu^0,\nu).
		\end{align*}
		Consequently, we get
		\begin{align*}
		\EE\left[\;\tmE^M\,\right]\leq C(T,L, \bu^0,\nu)\left(\PP(\Omega_{\tC,k})\right)^{1/2}+\tC k^{r/2}.
		\end{align*}
		Now we fix $\tC = k^{-{r}/{4}}$ from the beginning and define $\tOm_{M}:=\Omega_{M^{{{r}/{4}}},M^{-1}}$. To conclude, we take the limit for $M\to \infty$ and apply \cref{thm:convergence in probability},
		\begin{align*}
		\lim_{M\to \infty}\EE\left[\;\tmE^M\,\right]\leq C(T,L,\bu^0,\nu)\left(\lim_{M\to \infty}\PP(\tOm_{M})\right)^{1/2}+ \lim_{M\to \infty}\frac{1}{M^{r/4}}=0.
		\end{align*}
		This gives the assertion.	
	\end{proof}
	
	%
	

\end{document}